\documentclass[12pt,draftcls,onecolumn]{IEEEtran}
\usepackage{amsfonts}
\usepackage{amsmath}
\usepackage{amssymb}
\usepackage{epsfig}
\usepackage{graphics}
\usepackage{graphicx}
\usepackage{algorithm}
\usepackage{algorithmic}
\usepackage{multirow}
\usepackage{rotating}
\usepackage{subfigure}
\usepackage{color}

\epsfclipon
\newtheorem{theorem}{Theorem}[section]

\newtheorem{example}{Example}[section]

\newtheorem{lemma}[theorem]{Lemma}

\newtheorem{problem}{Problem}

\newtheorem{remark}{Remark}[section]

\input{tcilatex}
\begin{document}

\title{Spectral Design of Dynamic Networks via Local Operations}
\author{Victor M. Preciado,~\IEEEmembership{Member,~IEEE,} Michael M.
Zavlanos, \IEEEmembership{Member,~IEEE}, and Ali Jadbabaie,~%
\IEEEmembership{Senior~Member,~IEEE} \thanks{%
V.M. Preciado and A. Jadbabaie are with the Department of Electrical and
Systems Engineering at the University of Pennsylvania, Philadelphia, PA
19104 USA. (e-mail: preciado@seas.upenn.edu; jadbabai@seas.upenn.edu). } 
\thanks{%
M.M. Zavlanos is with the Department of Mechanical Engineering at the
Stevens Institute of Technology, Hoboken, NJ 07030 USA. (e-mail:
michael.zavlanos@stevens.edu). } \thanks{%
This work was supported by ONR MURI \textquotedblleft Next Generation
Network Science\textquotedblright\ and AFOSR \textquotedblleft Topological
and Geometric Tools for Analysis of Complex Networks\textquotedblright .}}
\maketitle

\begin{abstract}
Motivated by the relationship between the eigenvalue spectrum of a network
and the behavior of dynamical processes evolving in it, we propose a
distributed iterative algorithm in which a group of $n$ autonomous agents
self-organize the structure of their communication network in order to
control the network's eigenvalue spectrum. In our algorithm, we assume that
each agent has only access to a local (`myopic') view of the network around
it and that there is no centralized coordinator. In each iteration, agents
of the network perform a decentralized decision process in which agents
share limited information about their myopic vision of the network to find
the most beneficial edge addition/deletion from a spectral point of view. We
base our approach on a novel distance function defined in the space of
eigenvalue spectra that is written in terms of the spectral moments of the
Laplacian matrix. In each iteration, agents in the network run a greedy
algorithm to find the edge addition/deletion that minimized the spectral
distance to the desired spectrum. The spectral distance presents interesting
theoretical properties that allow an elegant and efficient distributed
implementation of the greedy algorithm using distributed consensus. Our
distributed algorithm is stable by construction, i.e., locally optimizes the
network's eigenvalue spectrum, and is shown to perform very well in
practice.
\end{abstract}


\section{Introduction}

\label{sec:introduction}

A wide variety of complex networks composed of autonomous agents are able to
display a remarkable level of self-organization despite the absence of a
centralized coordinator \cite{W62,H83}. For example, the intricate structure
of many biological, social and economic networks, emerges as the result of
local interactions between agents aiming to optimize their local utilities 
\cite{J08}. In most real cases, these agents have only access to myopic
information about the structure of the network around them. Despite the
limited information accessible to each agent, most of these
\textquotedblleft self-engineered\textquotedblright\ networks are able to
efficiently satisfy their functional requirements.

The behavior of many networked dynamical processes, such as information
spreading, synchronization, or decentralized coordination, is directly
related to the network eigenvalue spectra \cite{P08}. In particular, the
spectrum of the Laplacian matrix of a network plays a key role in the
analysis of synchronization in networks of nonlinear oscillators \cite%
{PC98,PV05}, as well as in the behavior of many distributed algorithms \cite%
{L97}, and decentralized control problems \cite{FM04,OM04}. Motivated by the
relationship between a network's eigenvalue spectrum and the behavior of
dynamical processes evolving in it, we propose a distributed iterative
algorithm in which a group of autonomous agents self-organize the structure
of their communication network in order to control the network's eigenvalue
spectrum. The evolution of the graph is ruled by a decentralized decision
process in which agents share limited information about their myopic vision
of the network to decide which network adjustment is most beneficial
globally.

Optimization of network eigenvalues has been studied by several authors in
both centralized \cite{GMS90,GB06,KM06} and decentralized settings \cite%
{DJ06}. In these papers, the objective is usually to find the weights
associated to the edges of a given network in order to optimize eigenvalues
of particular relevance, such as the Laplacian spectral gap or spectral
radius (i.e., the second smallest and largest eigenvalues of the Laplacian
matrix, respectively). In contrast to existing techniques, we propose a
distributed framework where we control the so-called spectral moments of the
Laplacian matrix by iteratively modifying the structure of the network. We
show that the benefits of controlling the spectral moments, instead of
individual eigenvalues, lies in lower computational cost and elegant
distributed implementation. The performance of our algorithm is illustrated
in nontrivial computer simulations.

The rest of this paper is organized as follows. In Section~\ref{sec:problem}%
, we review terminology and formulate the problem under consideration. In
Section~\ref{Moments from Metrics}, we introduce a decentralized algorithm
to compute the spectral moments of the Laplacian matrix from myopic views of
the network's structure. We also introduce a novel perturbation technique to
efficiently compute the effect of adding or removing edges on the spectral
moments. Based on these results, in Section~\ref{sec:controlmoments}, we
propose a distributed algorithm in which a group of autonomous agents modify
their network of interconnections to control of the spectral moments of a
network towards desired values. Finally, in Section~\ref{sec:simulations},
we illustrate our approach with several computer simulations.\bigskip


\section{Preliminaries \& Problem Definition}

\label{sec:problem}


\subsection{Eigenvalues of Graphs and their Spectral Moments}

Let $\mathcal{G}=\left( \mathcal{V},\mathcal{E}\right) $ be an undirected
graph, where $\mathcal{V}=\left\{ 1,\dots ,n\right\} $ denotes a set of $n$
nodes and $\mathcal{E}\subseteq \mathcal{V}\times \mathcal{V}$ denotes a set
of $e$ undirected edges. If $\left( i,j\right) \in \mathcal{E}$, we call
nodes $i$ and $j$ \emph{adjacent} (or first-neighbors), which we denote by $%
i\sim j$. We define the set of first-neighbors of a node $i$ as $\mathcal{N}%
_{i}=\{j\in \mathcal{V}:\left( i,j\right) \in \mathcal{E}\}.$ The \emph{%
degree} $d_{i}$ of a vertex $i$ is the number of nodes adjacent to it, i.e., 
$d_{i}=\left\vert \mathcal{N}_{i}\right\vert $.\footnote{%
We define by $|X|$ the cardinality of the set $X$.} An undirected graph is
called \emph{simple} if its edges are unweighted and it has no self-loops%
\footnote{%
A self-loop is an edge of the type $\left( i,i\right) $.}. A graph is \emph{%
weighted} if there is a real number associated with every edge. More
formally, a weighted graph $\mathcal{H}$ can be defined as the triad $%
\mathcal{H=}\left( \mathcal{V},\mathcal{E},\mathcal{W}\right) $, where $%
\mathcal{V}$ and $\mathcal{E}$ are the sets of nodes and edges in $\mathcal{H%
}$, and $\mathcal{W=}\left\{ w_{ij}\in \mathbb{R},\text{ for all }\left(
i,j\right) \in \mathcal{E}\right\} $ is the set of (possibly negative)
weights.

Graphs can be algebraically represented via matrices. The \emph{adjacency
matrix} of a simple graph $\mathcal{G}$, denoted by $A_{\mathcal{G}%
}=[a_{ij}] $, is an $n\times n$ symmetric matrix defined entry-wise as $%
a_{ij}=1$ if nodes $i$ and $j$ are adjacent, and $a_{ij}=0$ otherwise. Given
a weighted, undirected graph $\mathcal{H}$, the weighted adjacency matrix is
defined by $W_{\mathcal{H}}=\left[ w_{ij}\right] $, where $w_{ij}$ is the
weight associated to edge $\left( i,j\right) \in \mathcal{E}$ and $w_{ij}=0$
if $i$ is not adjacent to $j$. We define the \emph{degree matrix} of a
simple graph $\mathcal{G}$ as the diagonal matrix $D_{\mathcal{G}%
}=diag\left( d_{i}\right) $. We define the \emph{Laplacian matrix }$L_{%
\mathcal{G}}$ (also known as combinatorial Laplacian, or Kirchhoff matrix)
of a simple graph as $L_{\mathcal{G}}=D_{\mathcal{G}}-A_{\mathcal{G}}$. For
simple graphs, $L_{\mathcal{G}}$ is a symmetric, positive semidefinite
matrix, which we denote by $L_{\mathcal{G}}\succeq 0$ \cite{Big93}. Thus, $%
L_{\mathcal{G}}$ has a full set of $n$ real and orthogonal eigenvectors with
real nonnegative eigenvalues $0=\lambda _{1}\leq \lambda _{2}\leq ...\leq
\lambda _{n}$. Furthermore, the trivial eigenvalue $\lambda _{1}=0$ of $L_{%
\mathcal{G}}$ always admits a corresponding eigenvector $v_{1}=\left(
1,1,...,1\right) ^{T}$. The algebraic multiplicity of the trivial eigenvalue
is equal to the number of connected components in $\mathcal{G}$. The
smallest and largest nontrivial eigenvalues of $L_{\mathcal{G}}$, $\lambda
_{2}$ and $\lambda _{n}$, are called the spectral gap and spectral radius of 
$L_{\mathcal{G}}$, respectively.

Given an undirected (possibly weighted) graph $\mathcal{G}$, we denote its
Laplacian spectrum by $S\left( \mathcal{G}\right) =\left\{ \lambda
_{1},...,\lambda _{n}\right\} $, and define the $k$-th Laplacian spectral
moment of $\mathcal{G}$ as, \cite{Big93}: 
\begin{equation}
m_{k}\left( \mathcal{G}\right) \triangleq \frac{1}{n}\sum_{i=1}^{n}\lambda
_{i}^{k}.  \label{Moments Definition}
\end{equation}%
The following theorem states that an eigenvalue spectrum is uniquely
characterized by a finite sequence of moments:

\begin{theorem}
\label{Uniqueness from Moments}Consider two undirected (possibly weighted)
graphs $G_{1}$ and $G_{2}$ with Laplacian eigenvalue spectra $S\left(
G_{1}\right) =\{\lambda _{1}^{\left( 1\right) }\leq ...\leq \lambda
_{n}^{\left( 1\right) }\}$ and $S_{2}\left( G_{1}\right) =\{\lambda
_{1}^{\left( 2\right) }\leq ...\leq \lambda _{n}^{\left( 2\right) }\}$.
Then, $\lambda _{i}^{\left( 1\right) }=\lambda _{i}^{\left( 2\right) }$ for
all $1\leq i\leq n$ if and only if $m_{k}\left( G_{1}\right) =m_{k}\left(
G_{2}\right) $ for $0\leq k\leq n-1$.
\end{theorem}

\begin{proof}
In the Appendix.
\end{proof}

In the rest of this paper we will focus on the spectrum of the graph
Laplacian matrix $L_{\mathcal{G}}$ and its spectral moments, which we denote
by $m_{k}\left( L_{\mathcal{G}}\right) $. In this case, Theorem~\ref%
{Uniqueness from Moments}, implies that the Laplacian spectral moment of a
graph on $n$ nodes is uniquely characterize by the sequence of $n-1$
spectral moments $\left( m_{k}\left( L_{\mathcal{G}}\right) \right)
_{k=1}^{n-1}$. It is worth remarking that two nonisomorphic\footnote{%
Two simple graphs $\mathcal{G}_{1}$ and $\mathcal{G}_{2}$ with adjacency
matrices $A_{\mathcal{G}_{1}}$ and $A_{\mathcal{G}_{2}}$ are \emph{isomorphic%
} if there exists a permutation matrix $P_{n}$ such that $A_{\mathcal{G}%
_{1}}=P_{n}A_{\mathcal{G}_{2}}P_{n}^{T}$.} graphs $\mathcal{G}_{1}$ and $%
\mathcal{G}_{2}$ can present the same eigenvalue spectrum \cite{CDS98}, in
which case we say that $\mathcal{G}_{1}$ and $\mathcal{G}_{2}$ are
isospectral. In other words, the eigenvalue spectrum of a graph is not
enough to characterize its structure. On the other hand, as we shall show in
Section \ref{Moments from Metrics}, there are many interesting connections
between the structural features of a graph $\mathcal{G}$ and the spectral
moments of its Laplacian matrix, $m_{k}\left( L_{\mathcal{G}}\right) $.


\subsection{Local Structural Properties of Graphs}

\label{sec_local_structural_properties}

In this section we define a collection of structural properties that are
important in our derivations. A \emph{walk} of length $k$ from node $i_{1}$
to node $i_{k+1}$ is an ordered sequence of nodes $\left(
i_{1},i_{2},...,i_{k+1}\right) $ such that $i_{j}\sim i_{j+1}$ for $%
j=1,2,...,k$. One says that the walk \emph{touches} each of the nodes that
comprises it. If $i_{1}=i_{k+1}$, then the walk is closed. A closed walk
with no repeated nodes (with the exception of the first and last nodes) is
called a \emph{cycle}. Given a walk $p=\left( i_{1},i_{2},...,i_{k+1}\right) 
$ in a weighted graph $\mathcal{H}$ with weighted adjacency matrix $W_{%
\mathcal{H}}=\left[ w_{ij}\right] $, we define the weight of the walk as, $%
\omega \left( p\right) =w_{i_{1}i_{2}}w_{i_{2}i_{3}}...w_{i_{k}i_{k+1}}$.

We now define the concept of local neighborhood around a node. Let $\delta
\left( i,j\right) $ denote the \emph{distance} between two nodes $i$ and $j$
(i.e., the minimum length of a walk from $i$ to $j$). By convention, we
assume that $\delta \left( i,i\right) =0$. We define the $r$-th order
neighborhood $\mathcal{G}_{i,r}=(\mathcal{N}_{i,r},\mathcal{E}_{i,r})$\
around a node $i$ as the subgraph $\mathcal{G}_{i,r}\subseteq \mathcal{G}$
with node-set $\mathcal{N}_{i,r}\triangleq \left\{ j\in \mathcal{V}:\delta
\left( i,j\right) \leq r\right\} $, and edge-set $\mathcal{E}%
_{i,r}=\{(v,w)\in \mathcal{E}$ s.t. $v,w\in \mathcal{N}_{i,r}\}$. Given a
set of $k$ nodes $\mathcal{K}\subseteq \mathcal{V}$, we define $\mathcal{G}_{%
\mathcal{K}}$ as the subgraph of $\mathcal{G}$ with node-set $\mathcal{V}%
\left( \mathcal{G}_{\mathcal{K}}\right) =\mathcal{K}$ and edge-set $\mathcal{%
E}\left( \mathcal{G}_{\mathcal{K}}\right) =\{\left( i,j\right) \in \mathcal{E%
}$ s.t. $i,j\in \mathcal{K}\}$. We define $L_{\mathcal{G}}\left( \mathcal{K}%
\right) $ as the $k\times k$ \emph{submatrix} of $L_{\mathcal{G}}$ formed by
selecting the rows and columns of $L_{\mathcal{G}}$ indexed by $\mathcal{K}$%
. In particular, we define the Laplacian submatrix $L_{i,r}\triangleq L_{%
\mathcal{G}}\left( \mathcal{N}_{i,r}\right) $.

We say that a structural measurement is local with a certain radius $r$ if
it can be computed from the set of local neighborhoods $\{\mathcal{G}_{i,r}$%
, $i=1,...,n\}$. For example, the degree sequence of $\mathcal{G}$ is a
local structural measurement (with radius $1$), since we can compute the
degree of each node $i$ from the neighborhood $\mathcal{N}_{i,1}$. In
contrast, the eigenvalue spectrum of the Laplacian matrix is not a local
property, since we cannot compute the eigenvalues unless we know the
complete graph structure. One of the main contributions of this paper is to
propose a novel methodology to extract global information regarding the
Laplacian eigenvalue spectrum from the set of local neighborhoods.


\subsection{Spectral Metrics and Problem Definition}

\label{sec_problem_definition}

As discussed in Section~\ref{sec:introduction}, our goal is to propose a
distributed algorithm to control the eigenvalue spectrum of a multi-agent
network, via its spectral moments, by iteratively adding/removing edges in
the network; see Section~\ref{sec_local_structural_properties}. For this, we
define the following spectral distance between two graphs $G_{a}$ and $G_{b}$%
, with spectra $S_{a}=\{\lambda _{i}^{\left( a\right) }\}_{i=1}^{n}$ and $%
S_{b}=\{\lambda _{i}^{\left( b\right) }\}_{i=1}^{n}$, as\footnote{%
Note that $d_{M}$ is a distance in the space of eigenvalue spectra, but not
in the space of graphs, since we can find nonisomorphic graphs that are
isospectral.} 
\begin{equation}
d_{M}\left( S_{a},S_{b}\right) =\sum_{k=1}^{n-1}\left( m_{k}\left(
G_{a}\right) ^{1/k}-m_{k}\left( G_{b}\right) ^{1/k}\right) ^{2}.  \label{d_m}
\end{equation}%
According to Theorem \ref{Uniqueness from Moments}, two graphs are
isospectral if their first $n-1$ spectral moments coincide; thus, $d_{M}$ in
(\ref{d_m}) is in fact a distance function in the space of graph spectra. We
further define the \emph{spectral pseudometric}\footnote{%
A pseudometric is a generalization of distance in which two distinct points
(in our case, two distinct spectra) can have zero distance.}:%
\begin{equation}
d_{K}\left( S_{a},S_{b}\right) =\sum_{k=1}^{K}\left( m_{k}\left(
G_{a}\right) ^{1/k}-m_{k}\left( G_{b}\right) ^{1/k}\right) ^{2},  \label{d_k}
\end{equation}%
for $K<n-1$. The benefit of using the spectral pseudodistance versus other
spectral distances is due to the fact that, as we shall show in Section \ref%
{Moments from Metrics}, we can efficiently compute the first $K$ spectral
moments of the Laplacian matrix from the set of local Laplacian submatrices
with radius $\left\lfloor K/2\right\rfloor $, i.e., $\{L_{i,\left\lfloor
K/2\right\rfloor },i\in \mathcal{V}\}$. In other words, assuming that each
agent has access to the Laplacian submatrix associated to its neighborhood
with radius $r$, we shall show how to distributedly compute the first $2r+1$
Laplacian moments of the complete graph $L_{\mathcal{G}}$. With the notation
defined above, we can rigorously state the problem addressed in this paper
as follows:

\begin{problem}
\label{Main Problem}Given a desired spectrum $S^{\ast }=\left\{ \lambda
_{i}^{\ast }\right\} _{i=1}^{n}$, find a simple graph $\mathcal{G}^{\ast }$
such that its Laplacian eigenvalue spectrum, denoted by $S\left( \mathcal{G}%
^{\ast }\right) $, minimizes $d_{K}\left( S\left( \mathcal{G}^{\ast }\right)
,S^{\ast }\right) $.
\end{problem}

Finding a simple graph with a given (feasible \footnote{%
We say that an eigenvalue spectrum is feasible if there is a simple graph
whose Laplacian matrix presents that spectrum.}) eigenvalue spectrum is, in
general, a hard combinatorial problem, even in a centralized setting. In
this paper, we propose a distributed approximation algorithm to find a graph
with a spectrum `close to' $S^{\ast }$ in the $d_{K}$ pseudometric. In our
algorithm, a group of agents located at the nodes of a network iteratively
add/remove edges to drive the network's eigenvalue spectrum towards the
desired spectrum. In each iteration, the set of agents perform a
decentralized decision process to find the most beneficial edge
addition/deletion from the point of view of the global eigenvalue spectrum.

To formulate our algorithm, we first need to define the \emph{edit distance} 
$d_{E}\left( \mathcal{G}_{a},\mathcal{G}_{b}\right) $\ between two graphs $%
\mathcal{G}_{a}$ and $\mathcal{G}_{b}$, which is the minimum number of edge
additions plus edge deletions to transform $\mathcal{G}_{a}$ into a graph
that is isomorphic to $\mathcal{G}_{b}$. To approximately solve Problem~\ref%
{Main Problem} in a distributed way, we propose the following iteration to
determine a sequence of graphs $\{\mathcal{G}(t)\}_{t\geq 0}$, starting from
any graph $\mathcal{G}_{0}$: 
\begin{equation}
\begin{tabular}{rrl}
$\mathcal{G}(t+1)\triangleq $ & $\arg \min_{\mathcal{G}}$ & $d_{K}\left(
S\left( \mathcal{G}\right) ,S^{\ast }\right) $ \\ 
& $\text{s.t.}$ & $d_{E}\left( \mathcal{G}\left( t\right) ,\mathcal{G}%
\right) =1,$ \\ 
&  & $\lambda _{2}\left( \mathcal{G}\right) >0.$%
\end{tabular}
\label{Optimization Step}
\end{equation}
The resulting sequence of spectra $\{S\left( \mathcal{G}(t)\right) \}_{t\geq
0}$ converges to $S^{\ast }$ as $t$ grows. The constraint $d_{E}\left( 
\mathcal{G}\left( t\right) ,\mathcal{G}\left( t+1\right) \right) =1$
enforces only single edge additions or deletions at each iteration, while
the requirement $\lambda _{2}\left( \mathcal{G}\left( t\right) \right) >0$
enforces graph connectivity at all times, which will be necessary for the
distributed implementation in Section \ref{sec:controlmoments}. Note that
the Iteration (\ref{Optimization Step}) typically requires global knowledge
of the network structure. In this paper, we propose a computationally
efficient, distributed algorithm in which agents in the network solve (\ref%
{Optimization Step}) using only their local, myopic views of the network
structure. In particular, we shall show how the set of agents can compute,
in a distributed fashion, the effect of an edge addition/deletion on the
first $2r+1$ Laplacian moments. Furthermore, we shall also propose a
distributed algorithm to find the edge addition/deletion that minimizes the
resulting value of the spectral pseudodistance to $S^{\ast }$. Before we
describe the implementation details in Section~\ref{sec:controlmoments}, we
first provide the theoretical foundation for our approach in Section~\ref%
{Moments from Metrics}.

\begin{remark}[Convergence]
Several remarks are in order. First, note that it is not always possible to
find a simple graph that exactly match a given eigenvalue spectrum. Second,
the spectral pseudometric $d_{K}\left( S\left( \mathcal{G}\right) ,S^{\ast
}\right) $ may present multiple minima for a given $S^{\ast }$. These minima
could correspond, for example, to several isospectral graphs matching the
desired spectrum $S^{\ast }$ \cite{CDS98}. Therefore, iteration (\ref%
{Optimization Step}) may converge to different isospectral graphs depending
on the initial condition $\mathcal{G}_{0}$. Third, iteration (\ref%
{Optimization Step}) finds the most beneficial edge addition/deletion in
each time step, hence, this greedy approach may get trapped in a local
minimum. In practice, we observe that in our numerical simulations the
spectra of these local minima are remarkably close to those of the desired
spectrum.
\end{remark}


\section{Moment-Based Analysis of the Laplacian Matrix}

\label{Moments from Metrics}

In this section, we use tools from algebraic graph theory to compute the
spectral moments of the Laplacian matrix of $\mathcal{G}$ when only the set
of local Laplacian submatrices $\left\{ L_{i,r}\text{, }i\in \mathcal{V}%
\right\} $ is available. As a result of our analysis, we propose a
decentralized algorithm to compute a truncated sequence of Laplacian
spectral moments via a single distributed averaging. Furthermore, we also
present an efficient approach to compute the effect of adding or deleting an
edge in the Laplacian spectral moments of the graph. Particularly useful in
our derivations will be the following result from algebraic graph theory 
\cite{Big93}:

\begin{lemma}
\label{Diagonals as Walks}Let $\mathcal{H=}\left( \mathcal{V},\mathcal{E},%
\mathcal{W}\right) $ be a weighted graph with weighted adjacency matrix $W_{%
\mathcal{H}}=\left[ w_{ij}\right] $. Then%
\begin{equation*}
\left[ W_{\mathcal{H}}^{k}\right] _{ii}=\sum_{p\in P_{i,k}\left( \mathcal{H}%
\right) }\omega \left( p\right) ,
\end{equation*}%
where $P_{i,k}\left( \mathcal{H}\right) $ is the set of closed walks of
length $k$ starting and finishing at node $i$ in the weighted graph $%
\mathcal{H}$.
\end{lemma}


\subsection{Algebraic Analysis of Structured Matrices\label{Algebraic
Analysis}}

Consider the symmetric Laplacian matrix $L_{\mathcal{G}}$ of a simple graph $%
\mathcal{G=}\left( \mathcal{V},\mathcal{E}\right) $. We denote by $\mathcal{G%
}_{i,r}=\left( \mathcal{N}_{i,r},\mathcal{E}_{i,r}\right) $ the neighborhood
of radius $r$ around node $i$ and define the \emph{local Laplacian submatrix}
$L_{i,r}$, as the submatrix $L_{\mathcal{G}}\left( \mathcal{N}_{i,r}\right) $%
, formed by selecting the rows and columns of $L_{\mathcal{G}}$ indexed by
the set of nodes $\mathcal{N}_{i,r}$. By convention, we associate the first
row and column of the submatrix $L_{i,r}$ with node $i\in \mathcal{V}$,
which can be done via a simple permutation of rows and columns.\footnote{%
Notice that permuting the rows and columns of the Laplacian matrix does not
change the topology of the underlying graph.} For a simple graph $\mathcal{G}
$ with Laplacian matrix $L_{\mathcal{G}}$, we define $\mathcal{L}\left( 
\mathcal{G}\right) $ as the weighted graph whose adjacency matrix is equal
to $L_{\mathcal{G}}$. In other words, $\mathcal{L}\left( \mathcal{G}\right) $
has edges with weight $-1$ for $\left( i,j\right) \in \mathcal{E}\left( 
\mathcal{G}\right) $, $0$ for $\left( i,j\right) \not\in \mathcal{E}\left( 
\mathcal{G}\right) $, and $d_{i}$ for all self-loops $\left( i,i\right) $, $%
i\in \mathcal{V}\left( \mathcal{G}\right) $. We also define $\mathcal{H}%
_{i,r}$ as the weighted subgraph of $\mathcal{L}\left( \mathcal{G}\right) $
with node set $\mathcal{N}_{i,r}$, containing all the edges of $\mathcal{L}%
\left( \mathcal{G}\right) $ connecting pairs of nodes in $\mathcal{N}_{i,r}$
(including self-loops). Notice that, according to this definition, the
weighted adjacency matrix of $\mathcal{H}_{i,r}$ is equal to $L_{i,r}$.

In this paper, we assume that each agent in the network knows the structure
of its local neighborhood $\mathcal{G}_{i,r}$, for a fixed $r$. Therefore,
agent $i$ has access to the local Laplacian submatrix $L_{i,r}$. The
following results allows us aggregate information from the set of local
Laplacian submatrices, $\left\{ L_{i,r}\right\} _{i\in \mathcal{V}}$, to
compute a sequence of spectral moments of the (global) Laplacian matrix $L_{%
\mathcal{G}}$.

\begin{theorem}
\label{Moments from Subgraphs}Consider a simple graph $\mathcal{G}$ with
Laplacian matrix $L_{\mathcal{G}}$. Then, for a given radius $r$, the
Laplacian spectral moments can be written as%
\begin{equation}
m_{k}\left( L_{\mathcal{G}}\right) =\frac{1}{n}\sum_{i=1}^{n}\left[
L_{i,r}^{k}\right] _{11},  \label{Moments as Sum of Traces}
\end{equation}%
for $k\leq K=2r+1$.
\end{theorem}

\begin{proof}
Since the trace of a matrix is the sum of its eigenvalues, we can expand the 
$k$-th spectral moment of the Laplacian matrix as follows:%
\begin{eqnarray*}
m_{k}\left( L_{\mathcal{G}}\right) &=&\frac{1}{n}\text{Trace}\left( L_{%
\mathcal{G}}^{k}\right) \\
&=&\frac{1}{n}\sum_{i=1}^{n}\left[ L_{\mathcal{G}}^{k}\right] _{ii}
\end{eqnarray*}%
Therefore, since $L_{\mathcal{G}}$ is the weighted adjacency matrix of the
Laplacian graph $\mathcal{L}\left( \mathcal{G}\right)$, we have (from Lemma %
\ref{Diagonals as Walks})%
\begin{equation}
m_{k}\left( L_{\mathcal{G}}\right) =\frac{1}{n}\sum_{i=1}^{n}\sum_{p\in
P_{i,k}\left( \mathcal{L}\left( \mathcal{G}\right) \right) }\omega \left(
p\right) ,  \label{Moments as Walks in Laplacian Graph}
\end{equation}%
where the weights $\omega \left( p\right) $ are summed over the set of
closed walks of length $k$ starting at node $i$ in the weighted graph $%
\mathcal{L}\left( \mathcal{G}\right) $.

For a fixed value of $k$, closed walks of length $k$ in $\mathcal{L}\left( 
\mathcal{G}\right)$ starting at node $i$ can only touch nodes within a
certain distance $r\left( k\right) $ of $i$, where $r\left( k\right) $ is a
function of $k$ (see Fig. \ref{Walks in Gir}). In particular, for $k$ even
(resp. odd), a closed walk of length $k$ starting at node $i$ can only touch
nodes at most $k/2$ $\,$(resp. $\left\lfloor k/2\right\rfloor $) hops away
from $i$. Therefore, closed walks of length $k$ starting at $i$ are always
contained within the neighborhood of radius $\left\lfloor k/2\right\rfloor $%
. In other words, the neighborhood $\mathcal{G}_{i,r}$ of radius $r$
contains all closed walks of length up to $2r+1$ starting at node $i$.
Therefore, for $k\leq 2r+1 $, we have that 
\begin{equation*}
\sum_{p\in P_{i,k}\left( \mathcal{L} \right) }\omega \left( p\right)
=\sum_{p\in P_{1,k}\left( \mathcal{H}_{i,r} \right) }\omega \left( p\right) ,
\end{equation*}%
where $\mathcal{H}_{i,r} $ is the weighted graph whose adjacency matrix is
equal to the local Laplacian submatrix $L_{i,r}$ (notice that, by
convention, we associate the first row and column of $L_{i,r}$\ with node $i$%
). Therefore, according to Lemma \ref{Diagonals as Walks}, we have 
\begin{equation}
\sum_{p\in P_{1,k}\left( \mathcal{H}\left( L_{i,r}\right) \right) }\omega
\left( p\right) =\left[ L_{i,r}^{k}\right] _{1,1}.
\label{Walks as Diagonal of Sublaplacian}
\end{equation}%
Then, substituting (\ref{Walks as Diagonal of Sublaplacian}) into (\ref%
{Moments as Walks in Laplacian Graph}), we obtain the statement of our
Theorem.
\end{proof}

\begin{remark}[Distributed computation of spectral moments]
Since every node $i$ has access to its local neighborhood $\mathcal{G}_{i,r}$%
, it is possible to compute the first $2r+1$ moments via a simple
distributed averaging of the quantities $\{\left[ L_{i,r}^{k}\right]
_{1,1}\}_{i\in \mathcal{V}}$, \cite{L97}. This averaging efficiently
aggregates local pieces of local structural information (described by the
local Laplacian submatrices) to produce a truncated sequence of spectral
moments of the (global) Laplacian matrix. This is an useful result for the
analysis of complex networks for which retrieving the complete structure of
the network can be very challenging (in many cases, not even possible).
\end{remark}

\begin{figure}[tbp]
\centering\includegraphics[width=0.65\linewidth]{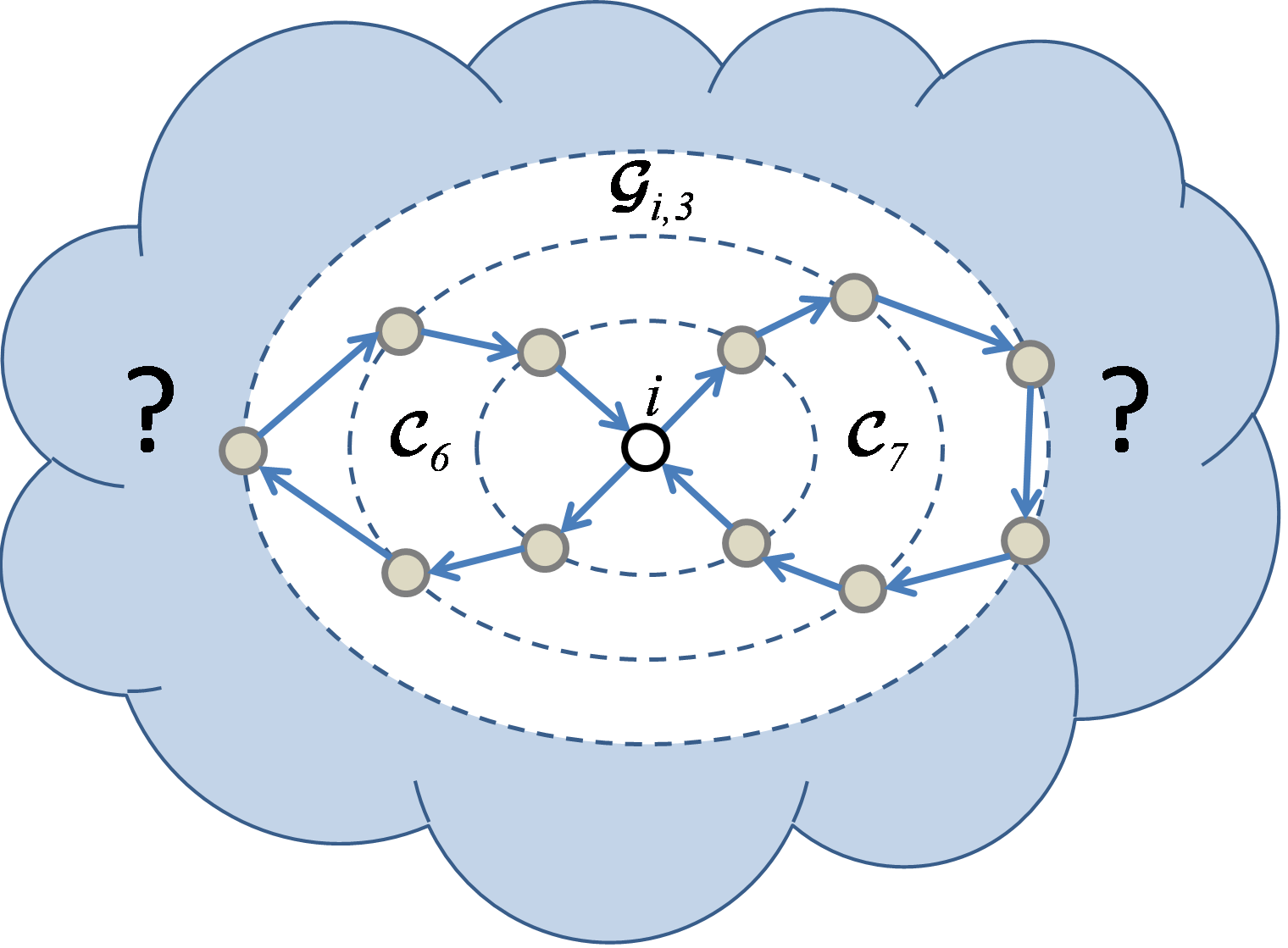}
\caption{Cycles $\mathcal{C}_{6}$ and $\mathcal{C}_{7}$, of lengths $6$ and $%
7\,$, in a neighborhood of radius $3$ around node \thinspace $i\,$.}
\label{Walks in Gir}
\end{figure}

Based on Theorem \ref{Moments from Subgraphs}, we propose a distributed
algorithm to compute a sequence of $2r+1$ spectral moments of $L_{\mathcal{G}%
}$ from local submatrices $L_{i,r}$, as described in Algorithm \ref%
{alg:moment_comp}. Note that, computing the spectral moments via (\ref%
{Moments as Sum of Traces})\ is much more efficient than computing these
moments via an explicit eigenvalue decomposition for many real-world
networks. In most real applications, the Laplacian matrix representing the
network structure is a sparse graph for which the number of nodes in the
neighborhood $\mathcal{N}_{i,r}$ is very small compared to $n$, for moderate
values of $r$.

\begin{algorithm}[t]
\caption{Decentralized moment computation} \label{alg:moment_comp}
\begin{algorithmic}[1]
\REQUIRE Local Laplacian submatrices $L_{i,r}$ for all nodes $i\in V$;%
\STATE Each node $i\in V$ computes a vector $\mathbf{\mu }%
_{i}\triangleq \left( \mu_{i,1},\mu_{i,2},\mu_{i,3},...,\mu_{i,2r+1}\right) ^{T}$, where $%
\mu_{i,k}\triangleq \left[ L_{i,r}^{k}\right]_{1,1} $;%
\STATE Using distributed averaging, compute the following
vector of averages:%
\begin{eqnarray*}
\mathbf{m}_{2r+1}\left( L_{\mathcal{G}}\right)  &\triangleq &\frac{1}{n}\sum_{i=1}^{n}%
\mathbf{\mu }_{i} \\
&=&\frac{1}{n}\sum_{i=1}^{n}\left( 0,\mu_{i,2},\mu_{i,3},...,\mu_{i,2r+1}\right)
^{T} \\
&=&\left( m_{1}\left( L_{\mathcal{G}}\right) ,m_{2}\left( L_{\mathcal{G}}\right) ,...,m_{2r+1}\left(
L_{\mathcal{G}}\right) \right) ^{T}.
\end{eqnarray*}
\end{algorithmic}
\end{algorithm}


\subsection{Moment-Based Perturbation Analysis}

\label{Perturbation Analysis}

In this section, we use spectral graph theory to compute the effect of
adding or deleting an edge on the spectral moments of the Laplacian matrix.
Traditionally, the effect of a matrix perturbation on the eigenvalue
spectrum is analyzed using eigenvalue perturbation techniques \cite{Wil65}.
In particular, the effect of adding a `small' perturbation matrix $\delta W$
to an $n\times n$ symmetric matrix $W$ with eigenvalue spectrum $\left\{
\sigma _{k}\right\} _{k=1}^{n}$ can be approximated, in the first-order, by 
\cite{Wil65} 
\begin{equation*}
\widetilde{\sigma }_{k}-\sigma _{k}\approx u_{k}^{T}~\delta W~u_{k},
\end{equation*}%
where $u_{k}$ is the eigenvector of $W$ associated with the eigenvalue $%
\sigma _{k}$, and $\left\{ \widetilde{\sigma }_{k}\right\} _{k=1}^{n}$ is
the eigenvalue spectrum of the perturbed matrix $W+\delta W$. In the case of
the Laplacian matrix, the perturbation matrix $\delta W$ corresponding to
the addition of an edge $\left( i,j\right) $ can be written as, $\delta
W=\left( e_{i}-e_{j}\right) \left( e_{i}-e_{j}\right) ^{T}$, where $e_{i}$
is the unit vector in the direction of the $i$-th coordinate. We denote by $%
\mathcal{G}+\left( i,j\right) $ the graph resulting of adding edge $\left(
i,j\right) $ to $\mathcal{G}$, and $\{\widetilde{\lambda }_{k}\}_{k=1}^{n}$
is the Laplacian spectrum of $\mathcal{G}+\left( i,j\right) $. Therefore,
adding edge $\left( i,j\right) $ perturbs the eigenvalues of the Laplacian
matrix as follows:%
\begin{eqnarray*}
\widetilde{\lambda }_{k}-\lambda _{k} &\approx &v_{k}^{T}\left(
e_{i}-e_{j}\right) \left( e_{i}-e_{j}\right) ^{T}v_{k} \\
&=&\left( v_{k,i}-v_{k,j}\right) ^{2},
\end{eqnarray*}%
where $v_{k}$ is the eigenvector of\ $L_{\mathcal{G}}$ associated to $%
\lambda _{k}$, and $v_{k,j}$ is its $j$-th component. Hence, the resulting
spectral radius can be approximated as%
\begin{equation*}
\widetilde{\lambda }_{1}\approx \lambda _{1}+\left( v_{1,i}-v_{1,j}\right)
^{2}.
\end{equation*}

Therefore, computing the effect of an edge addition on the spectral radius
using traditional perturbation techniques requires computation of the
dominant eigenvalue and eigenvector of $L_{\mathcal{G}}$, which is
computationally expensive for very large graphs. As an alternative to the
traditional analysis, we propose a novel approach, based on algebraic graph
theory, to compute the effect of structural perturbation on the spectral
moments of the Laplacian matrix $L_{\mathcal{G}}$ \emph{without explicitly
computing the eigenvalues or eigenvectors of }$L_{\mathcal{G}}$.
Furthermore, our approach can be efficiently implemented in a fully
decentralized manner.

In our derivations, we use the following result from algebraic graph theory:

\begin{lemma}
\label{Moments as Walks}Let $\mathcal{H=}\left( \mathcal{V},\mathcal{E},%
\mathcal{W}\right) $ be a weighted graph with weighted adjacency matrix $W_{%
\mathcal{H}}=\left[ w_{ij}\right] $. Then%
\begin{equation}
m_{k}\left( W_{\mathcal{H}}\right) =\frac{1}{n}\sum_{p\in P_{k}\left( 
\mathcal{H}\right) }\omega \left( p\right) ,  \label{Eq Moments as Walks}
\end{equation}%
where $P_{k}\left( \mathcal{H}\right) $ is the set of closed walks of length 
$k$ in the weighted graph $\mathcal{H}$.
\end{lemma}

\begin{proof}
This lemma is a consequence of Lemma \ref{Diagonals as Walks}. Specifically,
we have that%
\begin{eqnarray*}
m_{k}\left( W\left( \mathcal{H}\right) \right) &=&\frac{1}{n}\text{Trace}%
\left( W\left( \mathcal{H}\right) ^{k}\right) \\
&=&\frac{1}{n}\sum_{i\in \mathcal{V}}\left[ W\left( \mathcal{H}\right) ^{k}%
\right] _{i,i} \\
&=&\frac{1}{n}\sum_{i\in \mathcal{V}}\sum_{p\in P_{i,k}\left( \mathcal{H}%
\right) }\omega \left( p\right) \\
&=&\frac{1}{n}\sum_{p\in P_{k}\left( \mathcal{H}\right) }\omega \left(
p\right) ,
\end{eqnarray*}%
where $P_{k}\left( \mathcal{H}\right) \triangleq \cup_{i\in \mathcal{V}%
}P_{i,k}\left( \mathcal{H}\right) $ is the set of all closed walks of length 
$k$ in $\mathcal{H}$ (for any starting node $i\in \mathcal{V} $).
\end{proof}


\subsection{Perturbation on the Spectral Moments}

Consider a simple graph $\mathcal{G}$ with Laplacian matrix $L_{\mathcal{G}}$%
. We denote by $\mathcal{G}+\left( i,j\right) $ (resp. $\mathcal{G}-\left(
i,j\right) $) the graph resulting from adding (resp. removing) edge $\left(
i,j\right) $ to (resp. from) $\mathcal{G}$. Consider the sets of nodes $%
\mathcal{N}_{i,r}$ and $\mathcal{N}_{j,r}$ being within a radius $r$ from
node $i$ and node $j$, respectively. Let us define the following submatrices
indexed by the set of nodes in $\mathcal{N}_{i,r}\cup \mathcal{N}_{j,r}$:%
\begin{eqnarray*}
U_{r,\left( i,j\right) } &\triangleq &L_{\mathcal{G}}\left( \mathcal{N}%
_{i,r}\cup \mathcal{N}_{j,r}\right) , \\
U_{r,\left( i,j\right) }^{+} &\triangleq &L_{\mathcal{G}+\left( i,j\right)
}\left( \mathcal{N}_{i,r}\cup \mathcal{N}_{j,r}\right) , \\
U_{r,\left( i,j\right) }^{-} &\triangleq &L_{\mathcal{G}-\left( i,j\right)
}\left( \mathcal{N}_{i,r}\cup \mathcal{N}_{j,r}\right) .
\end{eqnarray*}
The following lemma allows us to efficiently compute the increment (resp.
decrement) in the Laplacian spectral moments of $\mathcal{G}$ due to the
addition (resp. removal) of edge $\left( i,j\right) $:

\begin{theorem}
\label{Moment Perturbation for Removal}Given a simple graph $\mathcal{G}$
with Laplacian matrix $L_{\mathcal{G}}$, the increment (decrement) in the $k$%
-th Laplacian spectral moment of a graph $\mathcal{G}$ due to the addition
or deletion of an edge $\left( i,j\right) $ can be written as%
\begin{equation}
m_{k}\left( L_{\mathcal{G}\pm \left( i,j\right) }\right) -m_{k}\left( L_{%
\mathcal{G}}\right) =\frac{1}{n}\left( \text{Trace}\left( U_{r,\left(
i,j\right) }^{\pm }\right) ^{k}-\text{Trace}\left( U_{r,\left( i,j\right)
}\right) ^{k}\right) ,  \label{Eq for Edge Removal}
\end{equation}%
for $k\leq 2r+1$.
\end{theorem}

\begin{proof}
Consider the weighted Laplacian graphs of $L_{\mathcal{G}}$, $L_{\mathcal{G}%
+\left( i,j\right) }$, and $L_{\mathcal{G}-\left( i,j\right) }$, which we
denote by $\mathcal{H}\triangleq \mathcal{L}\left( \mathcal{G}\right) $, $%
\mathcal{H}^{+}\triangleq \mathcal{L}\left( \mathcal{G}+\left( i,j\right)
\right) $ and $\mathcal{H}^{-}\triangleq \mathcal{L}\left( \mathcal{G}%
-\left( i,j\right) \right) $, respectively. (By definition, the adjacency
matrices of the Laplacian graphs are the Laplacian matrices of the graphs.)
Then, according to Lemma \ref{Moments as Walks}, we have that the $k$-th
spectral moments $m_{k}\left( L_{\mathcal{G}}\right) $, $m_{k}\left( L_{%
\mathcal{G}+\left( i,j\right) }\right) $ and $m_{k}\left( L_{\mathcal{G}%
-\left( i,j\right) }\right) $ can be written as weighted sums over the sets
of all closed walks of length $k$ in $\mathcal{H}$, $\mathcal{H}^{+}$, and $%
\mathcal{H}^{-}$, as follows,%
\begin{eqnarray*}
m_{k}\left( L_{\mathcal{G}}\right) &=&\frac{1}{n}\sum_{p\in P_{k}\left( 
\mathcal{H}\right) }\omega \left( p\right) , \\
m_{k}\left( L_{\mathcal{G}\pm \left( i,j\right) }\right) &=&\frac{1}{n}%
\sum_{p\in P_{k}\left( \mathcal{H}^{\pm }\right) }\omega \left( p\right) .
\end{eqnarray*}

We define $P_{k,r}^{\left( i,j\right) }\left( \mathcal{H}\right) $, $%
P_{k,r}^{\left( i,j\right) }\left( \mathcal{H}^{+}\right) $, and $%
P_{k,r}^{\left( i,j\right) }\left( \mathcal{H}^{-}\right) $ as the sets of
closed walks of length $k$ in, respectively, $\mathcal{H}$, $\mathcal{H}^{+}$%
, and $\mathcal{H}^{-}$ visiting only nodes in the set $\mathcal{N}%
_{i,r}\cup \mathcal{N}_{j,r}$. Then, we can split the summation in (\ref{Eq
Moments as Walks}) for the Laplacian matrices, as follows:%
\begin{eqnarray}
m_{k}\left( L_{\mathcal{G}}\right) &=&\frac{1}{n}\sum_{p\in P_{k,r}^{\left(
i,j\right) }\left( \mathcal{H}\right) }\omega \left( p\right) +\frac{1}{n}%
\sum_{p\in P_{k}\backslash P_{k,r}^{\left( i,j\right) }\left( \mathcal{H}%
\right) }\omega \left( p\right) ,  \label{Moment H} \\
m_{k}\left( L_{\mathcal{G}\pm \left( i,j\right) }\right) &=&\frac{1}{n}%
\sum_{p\in P_{k,r}^{\left( i,j\right) }\left( \mathcal{H}^{\pm }\right)
}\omega \left( p\right) +\frac{1}{n}\sum_{p\in P_{k}\backslash
P_{k,r}^{\left( i,j\right) }\left( \mathcal{H}^{\pm }\right) }\omega \left(
p\right) .  \label{Moment H+}
\end{eqnarray}

Notice that, as we illustrated in Fig. \ref{Walks in Gir}, none of the
closed walk of length $k\leq 2r+1$ touching node $i$ (resp. node $j$) can
leave the neighborhood $\mathcal{N}_{i,r}$ (resp. $\mathcal{N}_{j,r}$).
Therefore, all closed walks of length $k\leq 2r+1$ touching either node $i$
or $j$ (or both) are contained\footnote{%
We say that a walk is \emph{contained} in a set of nodes $N$ if it only
touches nodes in $N$.} in $\mathcal{N}_{i,r}\cup \mathcal{N}_{j,r}$. As a
consequence, none of the closed walks in $P_{k}\backslash P_{k}^{\left(
i,j\right) }\left( \mathcal{H}\right) $ or $P_{k}\backslash P_{k}^{\left(
i,j\right) }\left( \mathcal{H}^{\pm }\right) $ touches node $i$ or $j$.
Since addition/removal of edge $\left( i,j\right) $ does not influence those
walks not touching $i$ or $j$, we have that%
\begin{equation*}
\sum_{p\in P_{k}\backslash P_{k,r}^{\left( i,j\right) }\left( \mathcal{H}%
\right) }\omega \left( p\right) =\frac{1}{n}\sum_{p\in P_{k}\backslash
P_{k,r}^{\left( i,j\right) }\left( \mathcal{H}^{\pm }\right) }\omega \left(
p\right) .
\end{equation*}%
Thus, from (\ref{Moment H})\ and (\ref{Moment H+}) we have%
\begin{equation}
m_{k}\left( L_{\mathcal{G}^{\pm }\left( i,j\right) }\right) -m_{k}\left( L_{%
\mathcal{G}}\right) =\frac{1}{n}\sum_{p\in P_{k,r}^{\left( i,j\right)
}\left( \mathcal{H}^{^{\pm }}\right) }\omega \left( p\right) -\frac{1}{n}%
\sum_{p\in P_{k,r}^{\left( i,j\right) }\left( \mathcal{H}\right) }\omega
\left( p\right) .  \label{Moments as Long Walks}
\end{equation}

Since $P_{k,r}^{\left( i,j\right) }\left( \mathcal{H}\right) $ is the set of
all closed walks of length $k$ in $\mathcal{H}$ visiting nodes in the set $%
\mathcal{N}_{i,r}\cup \mathcal{N}_{j,r}$, we can apply Lemma \ref{Moments as
Walks} to obtain%
\begin{eqnarray}
\frac{1}{n}\sum_{p\in P_{k,r}^{\left( i,j\right) }\left( \mathcal{H}\right)
}\omega \left( p\right) &=&m_{k}\left( L_{\mathcal{G}}\left( \mathcal{N}%
_{i,r}\cup \mathcal{N}_{j,r}\right) \right)  \notag \\
&=&\frac{1}{n}\text{Trace}\left( U_{r,\left( i,j\right) }\right) ^{k}.
\label{U from Moments}
\end{eqnarray}%
Similarly, for $P_{k,r}^{\left( i,j\right) }\left( \mathcal{H}^{^{\pm
}}\right) $, we obtain%
\begin{equation}
\frac{1}{n}\sum_{p\in P_{k,r}^{\left( i,j\right) }\left( \mathcal{H}^{^{\pm
}}\right) }\omega \left( p\right) =\frac{1}{n}\text{Trace}\left( U_{r,\left(
i,j\right) }^{\pm }\right) ^{k}.  \label{U+ from Moments}
\end{equation}%
Finally, substituting (\ref{U from Moments}) and (\ref{U+ from Moments}) in (%
\ref{Moments as Long Walks}) provides us with the statement of our theorem.
\end{proof}

\begin{remark}[Computational cost]
According to Lemma~\ref{Moment Perturbation for Removal}, we can compute the
increment or decrement in the Laplacian spectral moments (up to order $2r+1$%
) by computing Trace$( U_{r,\left( i,j\right) }) ^{k}$ and Trace$(
U_{r,\left( i,j\right) }^{\pm }) ^{k}$. Notice that the sizes of $%
U_{r,\left( i,j\right) }$ and $U_{r,\left( i,j\right) }^{\pm }$ are $%
\left\vert \mathcal{N}_{i,r}\cup \mathcal{N}_{j,r}\right\vert $, which is
usually small for large sparse graphs (and moderate $r$).
\end{remark}



\section{Decentralized Control of Spectral Moments}

\label{sec:controlmoments}

In this section, we integrate the results developed in Section~\ref{Moments
from Metrics} with a novel technique for distributed connectivity
verification of edge additions or deletions in order to obtain a distributed
solution to Problem~\ref{Main Problem} in the form of (\ref{Optimization
Step}), as discussed in Section~\ref{sec_problem_definition}. %
%
This relies on the assumption that an agent at node $i$ is able to
communicate at time slot $t$ with all the agents in its first-order
neighborhood $\mathcal{N}_{i,1}\left( t\right) $ only.\footnote{%
Notice that, since $\mathcal{G}\left( t\right) $ is time-dependent, so are
the neighborhoods $\mathcal{G}_{i,r}\left( t\right) =\left( \mathcal{N}%
_{i,r}\left( t\right) ,\mathcal{E}_{i,r}\left( t\right) \right) $.}
Moreover, we also assume that every agent has only a myopic view of the
network structure. This means that at time slot $t$ agent $i\in \mathcal{V}$
only knows the topology of the neighborhood $\mathcal{G}_{i,r}\left(
t\right) $, within a particular radius $r$. This limits the set of possible
actions that every agent $i$ can take in every step of the iteration (\ref%
{Optimization Step}), to be \emph{local} edge additions of non-edges $\left(
i,j\right) \not\in \mathcal{E}\left( t\right) $ in $\mathcal{G}_{i,r}\left(
t\right) $ or local edge deletions of edges $\left( i,j\right) \not\in 
\mathcal{E}\left( t\right)$ in $\mathcal{G}_{i,1}\left( t\right) $.

In what follows, it will be useful to predetermine the \emph{master node}
for each edge $\left( i,j\right) \in \mathcal{E}\left( t\right) $, which can
be arbitrarily chosen from the set of nodes $\left\{ i,j\right\} $. The
notion of master node is useful to coordinate actions in our decentralized
algorithm. The agent located at the master node of $\left( i,j\right) $ is
the only one with the authority to decide if edge $\left( i,j\right) $ is
deleted. We denote by $\mathcal{D}_{i}(t)$ the set of edges having node $i$
as its master. In our simulations, we choose this set to be $\mathcal{D}%
_{i}(t)\triangleq \left\{ \left( i,j\right) \in \mathcal{E}%
(t)\;|\;i>j\right\} $.\footnote{%
Since the indices of all nodes in the network are distinct natural numbers,
this definition results in a unique assignment.} Similarly, it is useful to
predefine a master node for each nonedge\footnote{%
A pair of nodes $\left( i,k\right) $ is a nonedge of $\mathcal{G}$ if $%
\left( i,k\right) \not\in \mathcal{E}\left( \mathcal{G}\right) $.} $\left(
i,k\right) \not\in \mathcal{E}\left( t\right) $. The agent located at the
master node of the nonedge is the only one with the authority to decide if
edge $\left( i,k\right) $ is added to the network. We denote by $\mathcal{A}%
_{i}\left( t\right) $ the set of nonedges having node $i$ as its master. In
our case, we define this set as $\mathcal{A}_{i}\left( t\right) \triangleq
\left\{ \left( i,k\right) \not\in \mathcal{E}(t)\;|\;k\in \mathcal{N}%
_{i,r}(t)\text{ and}\;i>k\right\} $, where we limit node $k$ to be in $%
\mathcal{N}_{i,r}(t)$, since we are only considering local edge additions.



\subsection{Connectivity-Preserving Edge Deletions}

\label{sec:link_deletions}

In a centralized framework, network connectivity can be inferred from the
number of trivial eigenvalues of the Laplacian matrix. However, when only
local network information is available, only sufficient conditions for
connectivity can be verified. One such condition is the requirement that $%
\left\vert \mathcal{N}_{j}\left( t\right) \cap \mathcal{N}_{i,r}\left(
t\right) \right\vert >1$, which can be locally verified by agent $i$ with
knowledge of only $\mathcal{G}_{i,r}$. Since this condition is only
sufficient but not necessary for connectivity preservation, we need a
mechanism to check connectivity for those edges in the set 
\begin{equation*}
\mathcal{C}\left( t\right) =\left\{ \left( i,j\right) \in \mathcal{E}%
(t)\;:\;\left\vert \mathcal{N}_{j}\left( t\right) \cap \mathcal{N}%
_{i,r}\left( t\right) \right\vert =1\right\} .
\end{equation*}%
of critically connected edges, for which the sufficient condition does not
hold.

The proposed mechanism relies on a the concept of a maximum consensus. In
particular, consider a graph $\mathcal{G}(t)=\left( \mathcal{V},\mathcal{E}%
(t)\right) $ at time $t\geq 0$ and for any $(i,j)\in\mathcal{C}(t)$
associate a scalar variable $x_k^{(i,j)}(s)\in\mathbb{R}$ with every node $%
k\in\mathcal{V}$. Assume that the variables $x_k^{(i,j)}(s)$ are randomly
initialized and run the following maximum consensus update 
\begin{equation}  \label{eqn_simple_max_consensus}
x_{k}^{(i,j)}\left( s+1\right) =\max_{l\in \mathcal{N}_{k,1}-\{i,j\}} \{
x_{l}^{(i,j)}\left( s\right) \}
\end{equation}%
on the graph $\mathcal{G}(t)-(i,j)$ obtained by virtually disabling the link 
$(i,j)$ via blocking communication through it. Then, the network $\mathcal{G}%
(t)-(i,j)$ is almost surely connected if and only if the variables $%
x_{k}^{(i,j)}(s)$ for all $k\in \mathcal{V}$ converge to the common value $%
\max_{k}x_{k}^{(i,j)}\left( 0\right) $. Note that convergence in this case
takes place in finite time that is upper bounded by the diameter of the
network \cite{Cortes08}. This idea can be extended to simultaneous
verification of multiple link deletions in $\mathcal{C}\left( t\right)$. In
fact, since every edge is assigned a unique master agent, we can partition
the set $\mathcal{C}(t)$ in to $|\mathcal{V}|$ disjoint subsets $\mathcal{C}%
(t)\cap\mathcal{D}_i(t)$ for all $i\in\mathcal{V}$. This allows us to define
the sets $\mathcal{P}_{ki}=\{x_{k}^{\left( i,j\right) }\left( s\right)
:\left( i,j\right) \in\mathcal{C}(t)\cap\mathcal{D}_i(t) \}$ containing all
variables of agent $k$ that have as a master agent $i$. A simple schematic
of the proposed construction is shown in the following table: 
\begin{equation*}
\left.%
\begin{array}{c|ccc}
& \mathcal{C}\cap\mathcal{D}_1 & \mathcal{C}\cap\mathcal{D}_2 & \dots \\ 
\hline
1 & \mathcal{P}_{11}=\{x_1^{(1,j)} \} & \mathcal{P}_{12}=\{ x_1^{(2,j)} \} & 
\dots \\ 
2 & \mathcal{P}_{21}=\{ x_2^{(1,j)} \} & \mathcal{P}_{22}=\{ x_1^{(2,j)} \}
& \dots \\ 
\vdots & \vdots & \vdots & \ddots%
\end{array}%
\right.
\end{equation*}
Note that the second subscript $i$ in the set $\mathcal{P}_{ki}$ denotes
that master agent for the variables contained in $\mathcal{P}_{ki}$.
Therefore, agent $k$ initializes only those variables in the set $\mathcal{P}%
_{kk}$. Finally, stack all variables in the set $\mathcal{P}_{ki}$ in a
vector $\mathbf{x}_{ki}(s)\in \mathbb{R}^{|\mathcal{C}(t)\cap\mathcal{D}%
_{i}(t)|} $ and denote by $\left[ \mathbf{x}_{ki}(s)\right] _{\left(
i,j\right) }$ the scalar state associated with edge $\left( i,j\right) \in 
\mathcal{C}(t)\cap \mathcal{D}_{i}(t) $. %
Using the notation defined above, we can simultaneously verify connectivity
for all edges in $\mathcal{C}(t) $ by a high-dimensional consensus. For
this, every agent $k$ initializes randomly all vectors $\mathbf{x}%
_{ki}(0)\in \mathbb{R}^{|\mathcal{C}(t)\cap\mathcal{D}_i(t) |}$ for all
masters $i\in\mathcal{V}$ and updates the vectors $\mathbf{x}_{ki}(s)\in 
\mathbb{R}^{|\mathcal{C}(t)\cap\mathcal{D}_i(t) |}$ as follows:

\noindent \emph{Case I}: If $k$ is not a neighbor of the master agent $i$,
i.e., if $k\not\in \mathcal{N}_i$, then it updates the vectors $\mathbf{x}%
_{ki}(s)$ as 
\begin{equation}
\mathbf{x}_{ki}(s+1):=\max_{l\in \mathcal{N}_{k}(t)}\left\{ \mathbf{x}%
_{ki}(s),\mathbf{x}_{li}(s)\right\} ,  \label{eqn:deletion_1}
\end{equation}%
where the maximum is applied elementwise on the vectors.

\noindent \emph{Case II}: If $k$ is a neighbor of the master agent $i$,
i.e., if $k\in \mathcal{N}_i$, then it virtually removes link $(k,i)$ and
updates the entry $\lbrack \mathbf{x}_{ki}(s)]_{(k,i)}$ as 
\begin{equation}
\lbrack \mathbf{x}_{ki}(s+1)]_{(k,i)}:=\max_{l\in \mathcal{N}%
_{k}(t)\backslash \{i\}}\left\{ [\mathbf{x}_{ki}(s)]_{(k,i)},[\mathbf{x}%
_{li}(s)]_{(k,i)}\right\} ,  \label{eqn:deletion_2}
\end{equation}%
while for all other links $(j,i)\in \mathcal{C}(t)\cap\mathcal{D}_i(t)$ with 
$j\neq k$ it updates the entries $\lbrack \mathbf{x}_{ki}(s)]_{(j,i)}$ as 
\begin{equation}
\lbrack \mathbf{x}_{ki}(s+1)]_{(j,i)}:=\max_{l\in \mathcal{N}_{k}(t)}\left\{
[\mathbf{x}_{ki}(s)]_{(j,i)},[\mathbf{x}_{li}(s)]_{(j,i)}\right\}.
\label{eqn:deletion_3}
\end{equation}

\noindent \emph{Case III}: For the variables $\mathbf{x}_{kk}(s)$ for which $%
k$ is the master, it virtually removes the links $(k,j)\in\mathcal{C}(t)\cap%
\mathcal{D}_k(t)$ and updates the entries $\lbrack \mathbf{x}%
_{kk}(s)]_{(k,j)}$ as 
\begin{equation}
\lbrack \mathbf{x}_{kk}(s+1)]_{(k,j)}:=\max_{l\in \mathcal{N}%
_{k}(t)\backslash \{j\}}\left\{ [\mathbf{x}_{kk}(s)]_{(k,j)},[\mathbf{x}%
_{lk}(s)]_{(k,j)}\right\} .  \label{eqn:deletion_4}
\end{equation}

The high-dimensional consensus defined by (\ref{eqn:deletion_1})--(\ref%
{eqn:deletion_4}) converges in a finite time $\tau >0$ \cite{Cortes08}. When
this happens, node $k$ requests the entries $[\mathbf{x}_{ik}(\tau
)]_{(k,i)} $ from all its neighbors $i\in\mathcal{N}_k(t)$ for which $%
(k,i)\in \mathcal{C}(t)\cap\mathcal{D}_{k}(t)$ and compares them with $[%
\mathbf{x}_{kk}(\tau )]_{(k,i)}$. Since, violation of connectivity due to
deletion of $(k,j)$ would result in nodes $k$ and $i$ being in different
connected components, if $[\mathbf{x}_{kk}(s)]_{(k,i)}=[\mathbf{x}%
_{ik}(s)]_{(k,i)}$ then the network $\mathcal{G}\left( t\right) -\left(
k,i\right) $ would still remain connected. Hence, we can define the set 
\begin{equation}
\mathcal{S}_{k}(t)\triangleq \left\{ (k,i)\in \mathcal{C}(t) \cap\mathcal{D}%
_{k}(t) \; : \; [\mathbf{x}_{kk}(\tau)]_{(k,i)}=[\mathbf{x}%
_{ik}(\tau)]_{(k,i)}\right\} ,  \label{eqn:safe_deletions}
\end{equation}%
containing the edges in $\mathcal{C}(t) \cap\mathcal{D}_{k}(t)$ whose
removal does not disconnect the network.

\label{sec:connect_verif}

\begin{algorithm}[t]
\caption{Connectivity verification} \label{alg:connect_verif}
\begin{algorithmic}[1]
\REQUIRE ${\bf x}_{ij}(0)\in\mathbb{R}^{|\mathcal{C}(t)\cap\mathcal{D}_{j}(t)|}$ for all $i,j\in\mathcal{V}$ ;%
\FOR{$s=1:\tau$}%
\STATE Update ${\bf x}_{ij}(s+1)$ by (\ref{eqn:deletion_1})--(\ref{eqn:deletion_4});%
\ENDFOR
\STATE Compute $\mathcal{S}_i(t)$ by (\ref{eqn:safe_deletions});%
\end{algorithmic}
\end{algorithm}

\subsection{Most Beneficial Local Action}

\label{sec:best_local_action}

To solve Problem~\ref{Main Problem} via the iterative algorithm proposed in (%
\ref{Optimization Step}), we need to add or delete an edge $\left(
i,j\right) $ that minimizes the spectral pseudometric $d_{K}(S(\mathcal{G}%
_{\pm (i,j)}(t)),S^{\ast })$ at every time step $t$. For this, let $SD_i(t)
\triangleq d_{K}\left( S\left( \mathcal{G}(t) \right) ,S^{\ast }\right)$
denote a local copy of the spectral distance of the graph $\mathcal{G}(t)$
that is available to agent $i$, so that initially $SD_i(0)=SD(0)$ for all
agents $i\in\mathcal{V}$. The quantity $SD(0)$ can be computed in a
distributed way by means of distributed averaging, according to Theorem~\ref%
{Moments from Subgraphs}. Then, the key idea is that every master agent $i$
computes the spectral distance $SD_{\pm (i,j)}(s) \triangleq d_{K}\left(
S\left( \mathcal{G}_{\pm (i,j)}(s) \right) ,S^{\ast }\right)$ resulting from
adding a link $(i,j)\in\mathcal{A}_i(t)$ or deleting a link $(i,j)\in%
\mathcal{S}_i(t)$. Computation of this distance relies on Theorem \ref{Moment
Perturbation for Removal} and
requires that agent $i$ has knowledge of the structure of its neighborhoods $%
\mathcal{G}_{i,r}$ only, for $r=\lfloor K/2\rfloor$. For all possible local
edge additions or deletions, master agent $i$ determines the most beneficial
one 
\begin{equation*}
(i,j_i^*(t)) \triangleq \underset{(i,j)\in \mathcal{A}_{i}(t)\cup\mathcal{S}%
_i(t) }{\mathrm{argmin}}\; \left\{SD_{\pm (i,j)}(t) - SD_i(t)\right\}.
\end{equation*}%
Note that the minimization above may result in multiple edges having the
same optimal value. Such ties can be broken via, e.g., a coin toss. Then,
the largest decrease in the error associated with the most beneficial edge $%
(i,j_i^*(t))$ becomes: 
\begin{equation*}
SD_{i}(t) \triangleq \left\{%
\begin{array}{ll}
SD_{\pm(i,j_i^*)}(t), & \text{if $\min_{(i,j)\in \mathcal{A}_{i}(t)\cup%
\mathcal{S}_i(t)}\{SD_{\pm(i,j)}(t)-SD_i(t)\}\leq 0$} \\ 
D, & \text{otherwise}%
\end{array}%
\right..
\end{equation*}
for a large constant $D>0$. In other words, $SD_{i}(t)$ is nontrivially
defined only if the exists a link adjacent to node $i$ that if added or
deleted decreases the error function $SD(t)$. Otherwise, a large value $D>0$
is assigned to $SD_{i}(t)$ to indicate that this action is not beneficial to
the final objective. Finally, for each node $i$, we initialize the state
vector 
\begin{equation*}
\mathbf{b}_{i}(0) \triangleq \left[ i\; j_i^*(t)\;SD_i(t) \;\mathbf{m}%
(i,j_i^*(t))\right] ^{T},
\end{equation*}%
containing the best local action $(i,j_i^*(t))$, the associated spectral
pseudodistance $SD_i(t)$, and the vector of resulting moments 
\begin{equation*}
\mathbf{m}(i,j_i^*(t)) \triangleq \left[ m_{k}\left(S\left( \mathcal{G}%
_{\pm(i,j_i^*)}(t) \right) \right)\right] _{k=1}^{K}.
\end{equation*}%
In the following section, we discuss how to compare all local actions $%
\mathbf{b}_{i}(t)$ for all nodes $i\in \mathcal{V}$ to find the best global
action that minimizes the spectral pseudometric.

\subsection{From Local Information to Global Action}

\label{sec:best_global_action}

\begin{algorithm}[t]
\caption{Globally most beneficial action} \label{alg1}
\begin{algorithmic}[1]
\REQUIRE ${\bf b}_{i}(0)\triangleq[i \; j_i^*(t) \; SD_i(t) \; {\bf m}(i,j_i^*(t))]^{T}$;%
\FOR{$s=1:\tau$}%
\STATE ${\bf b}_{i}(s+1) := {\bf b}_{j}(s)$, with $j=\max\{{\rm
argmin}_{k\in\mathcal{N}_i(t)}\{[{\bf b}_{i}(s)]_3,[{\bf b}_{k}(s)]_3\}$;%
\ENDFOR
\IF{$[{\bf b}_{i}(\tau)]_3<D$}%
\STATE Update $\mathcal{N}_i(t+1)$, ${\bf m}_i(t+1)$ and $SD_i(t+1)$ according to (\ref{eqn:update_neighbors_1})--(\ref{eqn:update_CME});%
\ELSIF{$[{\bf b}_{i}(\tau)]_3=D$}%
\STATE No beneficial action. Algorithm has converged;%
\ENDIF%
\end{algorithmic}
\end{algorithm}

In order to obtain the overall most beneficial action, all local actions
need to be propagated in the network and compared against each other. For
this, every agent $i$ communicates with its neighbors and updates 
its desired action $\mathbf{b}_{i}(s)$ with the action $\mathbf{b}_{j}(s)$
corresponding to the node $j$ that contains the smallest distance to the
target moments $[\mathbf{b}_j(s)]_{3}\triangleq SD_i(t)$, i.e., 
\begin{eqnarray*}
\mathbf{b}_{i}(s+1) &=&\mathbf{b}_{j}(s),\text{ where} \\
j &=&\mathrm{argmin}_{k\in \mathcal{N}_{i}(t)}\{[\mathbf{b}_{i}(s)]_{3},[%
\mathbf{b}_{k}(s)]_{3}\}.
\end{eqnarray*}%
In case of ties in the distances to the targets $[\mathbf{b}_{j}(s)]_{3}$,
then the node with the largest index is selected (line 2, Alg.~\ref{alg1}).
Note that line 2 of Alg.~\ref{alg1} is essentially a minimum consensus
update on the entries $[\mathbf{b}_{i}(s)]_{3}$ and will converge to a
common outcome for all nodes in finite time $\tau>0$, when they have all
been compared to each other. 
When the consensus has converged, if there exists a node whose desired
action decreases the distance to the target moments, i.e., if $[\mathbf{b}%
_i(s)]_{3}<D$ (line 4, Alg.~\ref{alg1}), then Alg.~\ref{alg1} terminates
with a greedy action and node $i$ updates its set of neighbors $\mathcal{N}%
_{i}(t+1)$ and vector of moments $\mathbf{m}_{i}(t+1)$ (line 5, Alg.~\ref%
{alg1}). If the optimal action is a link addition, i.e., if $[\mathbf{b}%
_i(\tau)]_{2}\not\in \mathcal{N}_{i}(t)$, then 
\begin{equation}
\mathcal{N}_{i}(t+1):=\mathcal{N}_{i}(t)\cup \left\{ \lbrack \mathbf{b}%
_{i}(\tau)]_{2}\right\} .  \label{eqn:update_neighbors_1}
\end{equation}%
On the other hand, if the optimal action is a link deletion, i.e., if $[%
\mathbf{b}_{i}(\tau)]_{2}\in \mathcal{N}_i(t)$, then 
\begin{equation}
\mathcal{N}_i(t+1):=\mathcal{N}_i(t)\backslash \left\{ \lbrack \mathbf{b}%
_{i}(\tau)]_{2}\right\} .  \label{eqn:update_neighbors_2}
\end{equation}%
In all cases, the moments and error function are updated by 
\begin{equation}
\mathbf{m}_{i}(t+1):=\left[[\mathbf{b}_{i}(\tau)]_{4}\dots [\mathbf{b}%
_{i}(\tau)]_{4+K}\right]^T  \label{eqn:update_moments}
\end{equation}%
and 
\begin{equation}
SD_{i}(t+1):=[\mathbf{b}_{i}(\tau)]_{3},  \label{eqn:update_CME}
\end{equation}%
respectively. Finally, if all local desired actions increase the distance to
the target moments, i.e., if $[\mathbf{b}_{i}(\tau)]_{3}=D$ (line 6, Alg.~%
\ref{alg1}), then no action is taken and the algorithm terminates with a
network topology with almost the desired spectral properties. This is
because no action exists that can further decrease the distance to the
target moments.

\subsection{Synchronization}

\label{sec:synchronization}

Communication time delays, packet losses, and the asymmetric network
structure, may result in runs of the algorithm starting asynchronously,
outdated information being used for future decisions, and consequently,
nodes reaching different decisions for the same run. In the absence of a
common global clock, the desired synchronization is ideally \emph{event
triggered}, where by a triggering event we understand the time instant that
messages are transmitted and received by the nodes. For an implementation of
such a scheme see \cite{ZP08}.



\section{Numerical Simulations}

\label{sec:simulations}

In the following numerical examples, we illustrate the performance and
limitations of our iterative graph process. The objective of our simulations
is to find a graph whose Laplacian spectral moments match those of a desired
spectrum. In each example, we analyze the performance of our algorithm and
study the spectral and structural properties of the resulting graph.

\begin{example}[Star vs. Two-Star Networks]
\label{exmp:stars}In our first two simulations, we try to find graphs that
match the spectral moments of (i) a star graph and (ii) a two-star graph
(Fig. \ref{fig:2_star_network}). The Laplacian spectral moments of a star
network with $10$ nodes are: $\left( m_{k}\right) _{k=1}^{5}=\left(
1.8,10.8,100.8,1000.8,10000.8\right) $. Starting with a random graph on $10$
nodes, we run our distributed algorithm to iteratively add and delete edges
that minimize the spectral pseudodistance. We observe, in Fig. \ref%
{fig:star_network}, that the spectral pseudodistance evolves towards zero in
45 steps. We also verify that, although we are only controlling the first
five spectral moments of the Laplacian matrix, the resulting network
structure is exactly the desired star topology. This indicates that a star
graph is an extreme case in which the graph topology is uniquely defined by
their first five Laplacian spectral moments.

In our second simulation, we consider the two-star network with 20 nodes in
Fig.~\ref{fig:2_star_network} (a). The Laplacian spectral moments of this
graph are $\left( m_{k}\right) _{k=1}^{5}=\left(
1.9,12.8,133.6,1480,16590\right) $. We observe in Fig. \ref{fig:star_network}
how, after running our iterative algorithm for 94 iterations, our graph
process stops in a graph topology with a spectral pseudodistance very close
to zero (in particular, $5.2e-2$). The resulting topology, represented in
Fig. \ref{fig:2_star_network} (b), is very close to the desired two-star
network. This topology is a local minima of our evolution process because we
could transform it into our optimal two-star graph by two simple operations:
(1) Adding an edge connecting nodes $u$ and $v$ (Fig. \ref%
{fig:2_star_network} (b)), and (2) removing edge $\left( u,w\right) $. On
the other hand, one can verify that step (1) would increase the spectral
pseudodistance; therefore, our greedy evolution process does not follow this
two-steps path. Despite this limitation, our final topology is remarkably
close to the two-star network and their eigenvalue spectra are very similar,
as shown in Fig.~\ref{fig:CDFs}.

\begin{figure}[tbp]
\centering\includegraphics[width=0.7\linewidth]{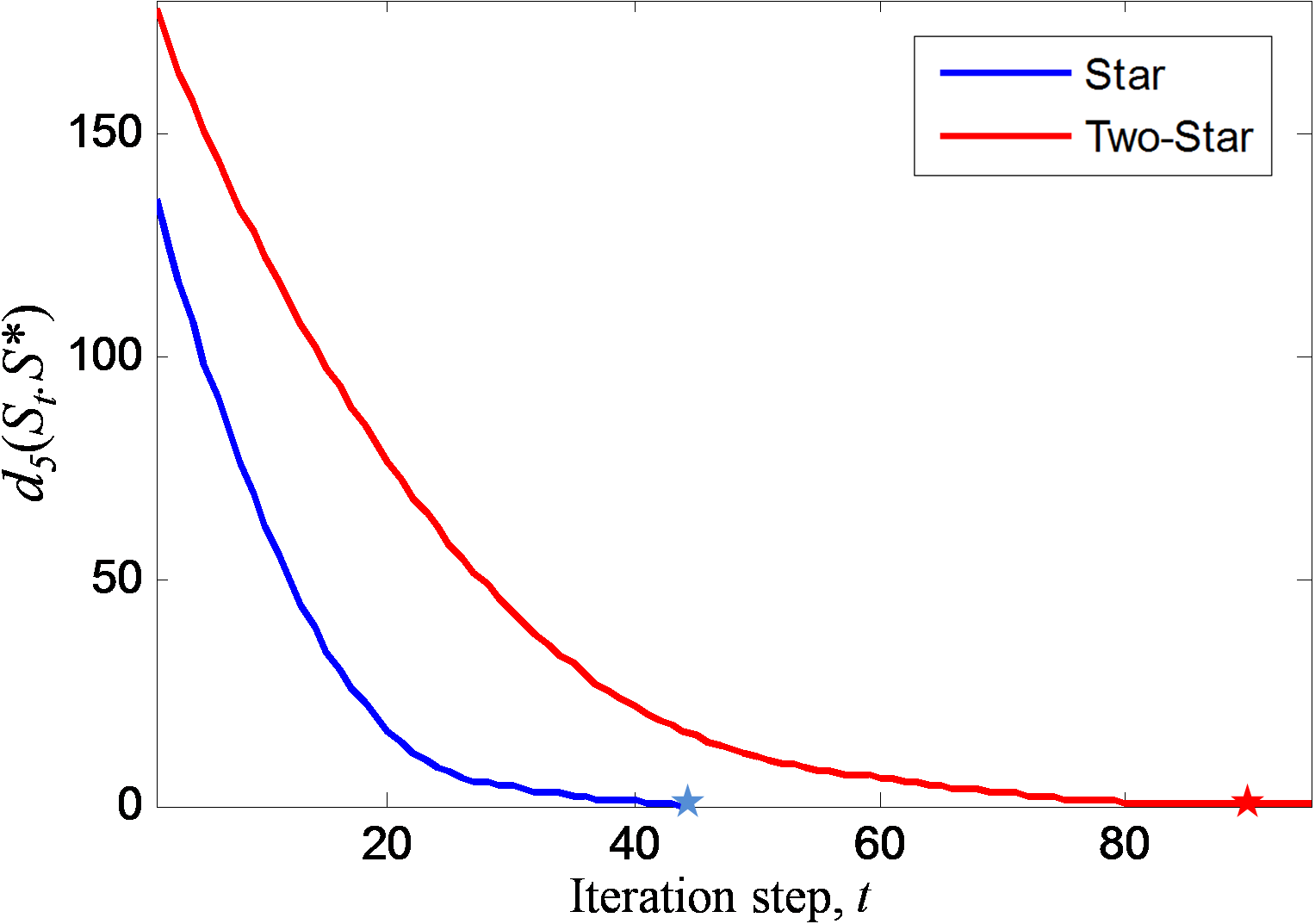}
\caption{Convergence of the spectral pseudodistance $d_{k}(S_{t},S^{\ast })$
for the star graph (blue plot) and the two-stars graph (red plot), where $%
S^{\ast }$ is the spectrum of the desired graph and $S_{t}$ is the spectrum
of $\mathcal{G}_{t}$.}
\label{fig:star_network}
\end{figure}
\end{example}

\begin{figure}[tbp]
\centering\includegraphics[width=0.7\linewidth]{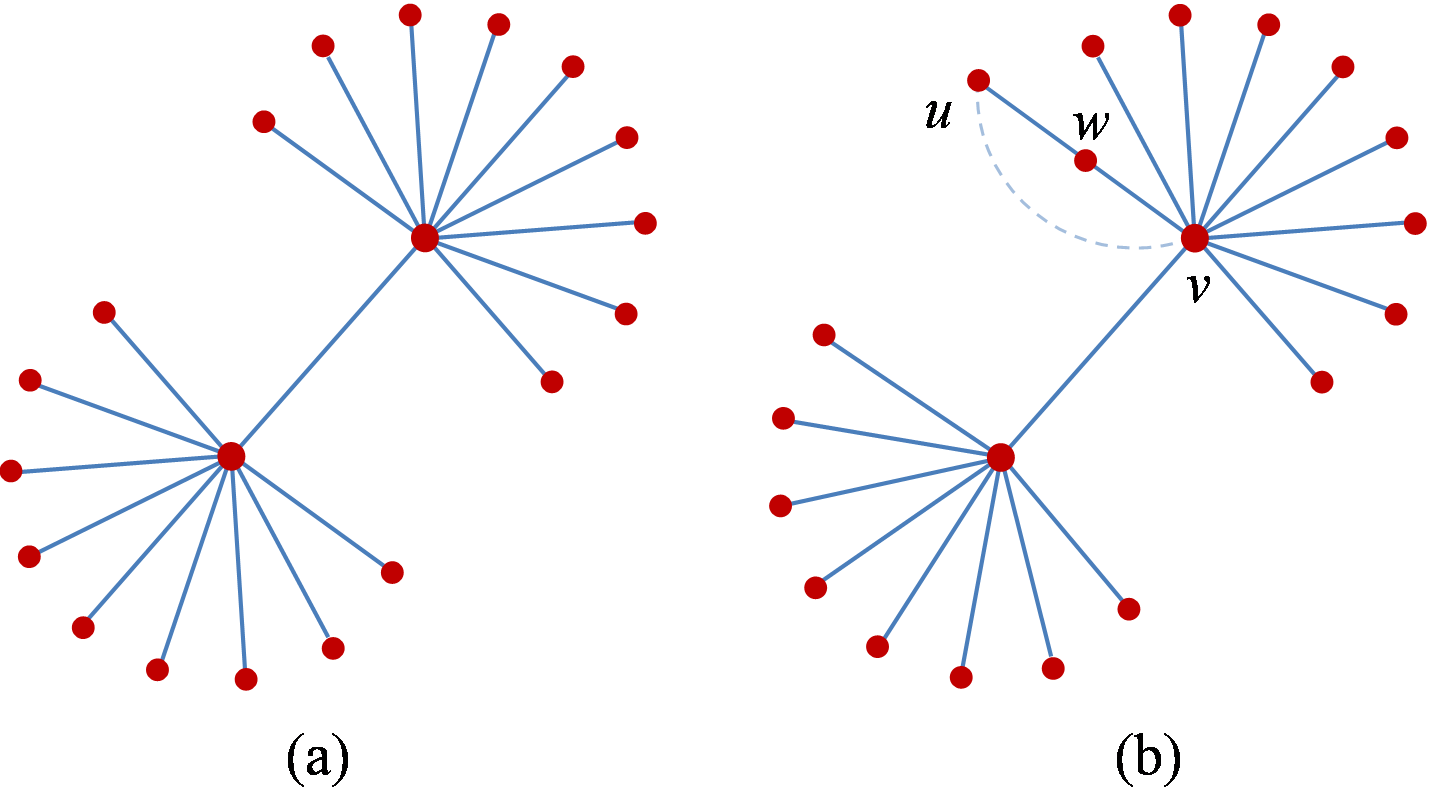}
\caption{Structures of the two-stars network (a) and the network returned by
our algorithm (b).}
\label{fig:2_star_network}
\end{figure}

\begin{figure}[tbp]
\centering\includegraphics[width=0.7\linewidth]{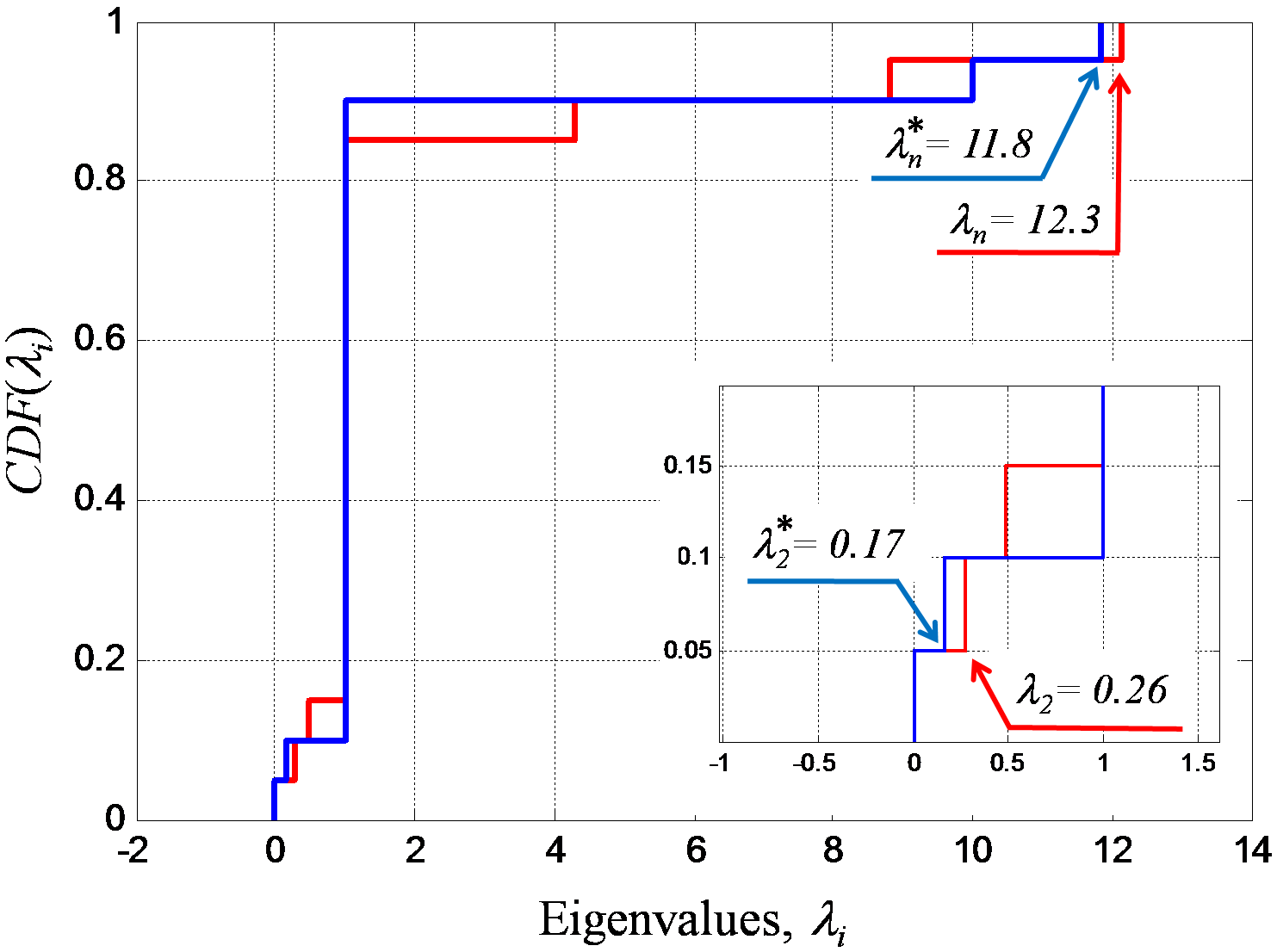}
\caption{Empirical cumulative distribution functions for the eigenvalues of
the two-stars graph (blue) and the graph returned by our algorithm (red).
The subgraph in the lower right corner shows the CDF's around the origin.}
\label{fig:CDFs}
\end{figure}

\begin{example}[Chain vs. ring networks]
\label{exmp:chains_rings}In the next two simulations, we try to find graphs
that match the spectral moments of (i) a ring graph and (ii) a chain graph.
Starting from a random graph, we run our iterative algorithm to match the
spectral moments of a chain graph with 20 nodes, $\left( m_{k}\right)
_{k=1}^{5}=\left( 1.9,5.6,18.4,63.6,226.4\right) $. In this case, the
spectral pseudodistance converges to zero in finite time and the final
topology is exactly the desired chain graph. On the other hand, if we try to
match the spectral moments of the ring graph in Fig. \ref{fig:ring_networks}
(a), with $\left( m_{k}\right) _{k=1}^{5}=\left( 2,6,20,70,252\right) $, an
exact reconstruction is very difficult to achieve. In Fig.~\ref%
{fig:ring_networks} (b), we depict the graph returned by our algorithm,
after 83 iterations. Note that since we are only allowing local structural
modifications in our graph process, it is hard for our algorithm to
replicate long cycles in the graph. On the other hand, although the
structure of the resulting network is not the desired ring graph, its
eigenvalue spectrum is remarkably close to that of a ring, as we can see in
Fig.~\ref{fig:CDF_ring}.
\end{example}

\begin{figure}[tbp]
\centering\includegraphics[width=0.7\linewidth]{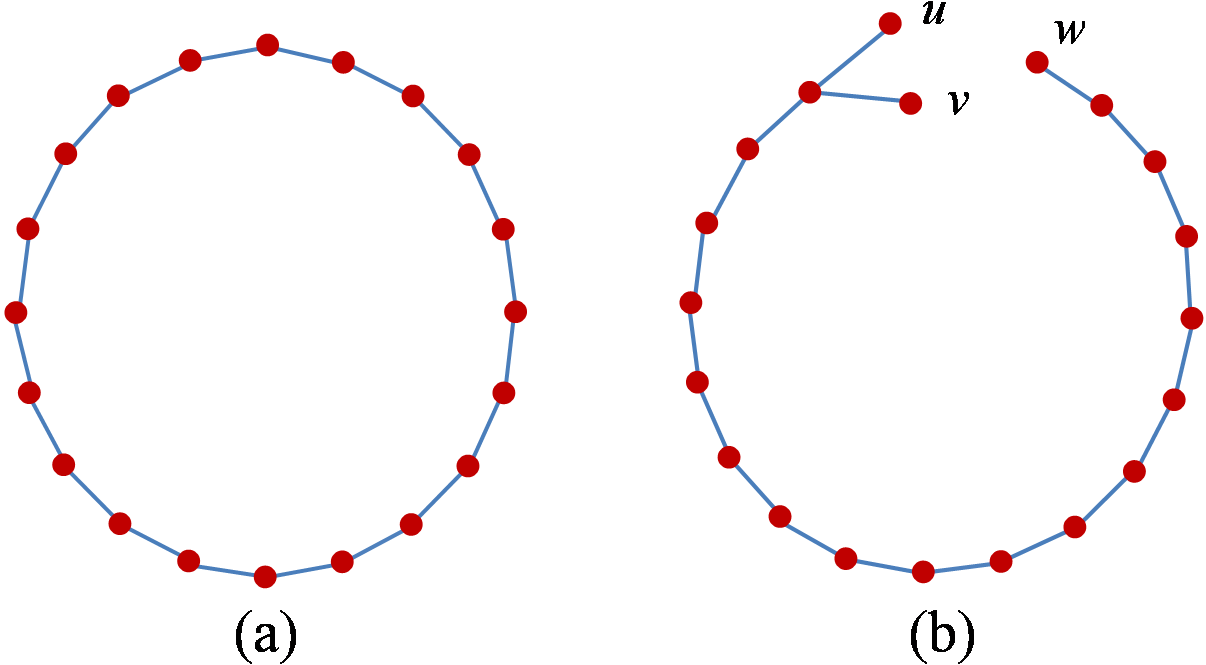}
\caption{In (a) we observe a ring graph with 20 nodes. The topology returned
by our iterative algorithm can be observed in (b).}
\label{fig:ring_networks}
\end{figure}

\begin{figure}[tbp]
\centering\includegraphics[width=0.7\linewidth]{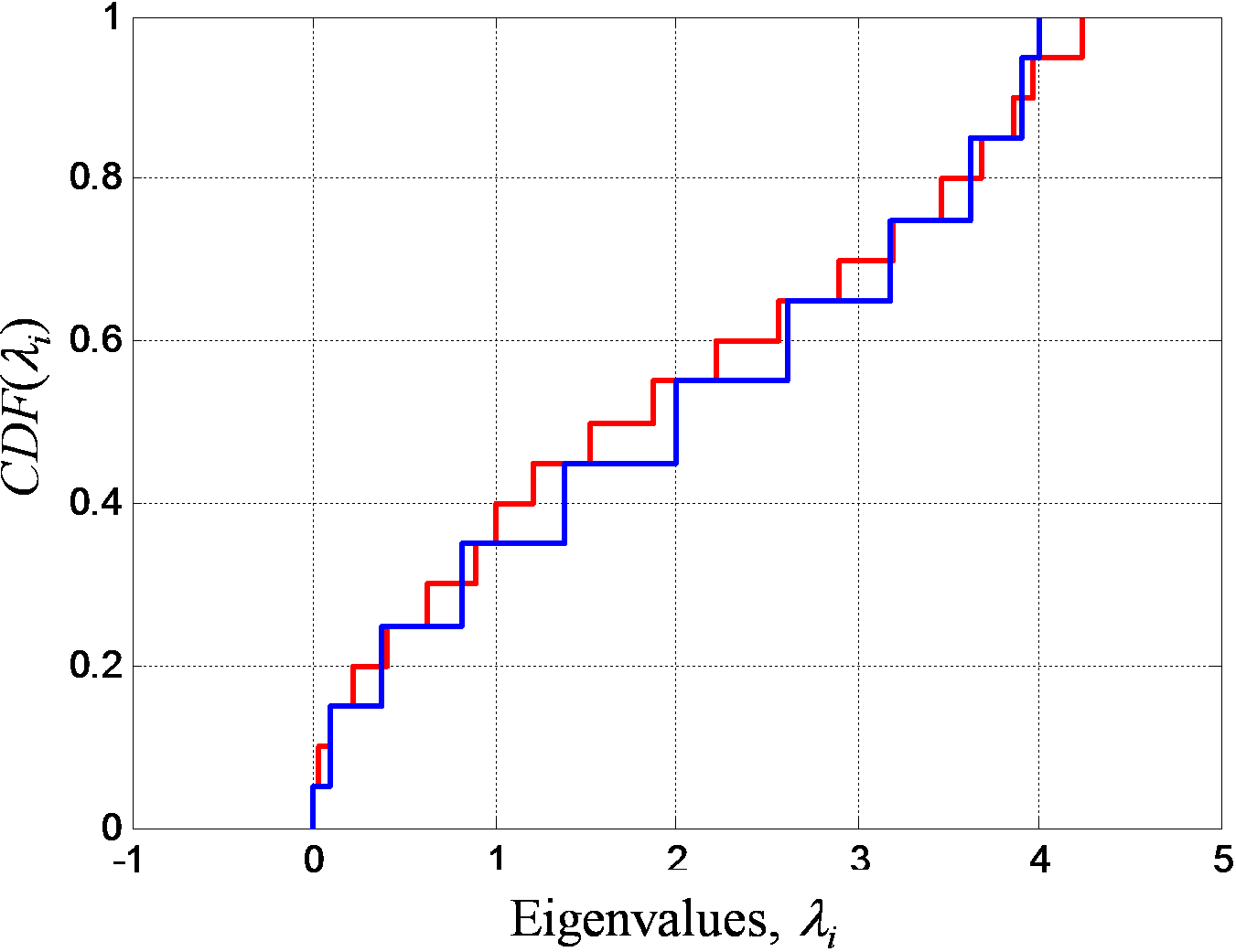}
\caption{Empirical cumulative distribution of eigenvalues for the ring graph
with 20 nodes (blue plot) and the topology returned by our algorithm in Fig.~%
\protect\ref{fig:ring_networks} (b) (red plot).}
\label{fig:CDF_ring}
\end{figure}

The above examples illustrate two limitations of our algorithm, namely, the
existence of local minima in the graph evolution process and the inability
of our algorithm to recover long cycles. Despite these limitations, our
algorithm is able to find graph topologies with eigenvalue spectra
remarkably close to the desired ones by matching five spectral moments only.
Furthermore, the resulting topologies are structurally very similar to the
desired ones, indicating that the spectral moments of the Laplacian matrix
contains rich information about the structure of a network. In the next two
examples, we show how our algorithm is also able to efficiently generate
graphs matching the spectral properties of two popular synthetic network
models: the Small-World \cite{WS98} and the Scale-Free \cite{BA99} networks.

\begin{example}[Small-Worlds]
\label{exmp:small_world}The small-world model was proposed by Watts and
Strogatz \cite{WS98} to generate networks with high clustering\footnote{%
The clustering coefficient of a network is a measure of the number of
triangles present in the network.} coefficients and small average distance.
We can generate a small-world network by following these steps: (1) take a
ring graph with $n$ nodes, (2) connect each node in the ring to all its
neighborhoods within a distance $k$, and (3) add random edges with a
probability $p$. In this example, we generate a small-world network with $%
n=40$, $k=2$, and $p=3/n$. The first three spectral moments of a random
realization of this network are $\left( m_{k}\right) _{k=1}^{3}=\left(
6.55,51.9,457\right) $. Then, we run our algorithm to generate a graph whose
first three spectral moments are close to those of the small-world network.
After running our algorithm for $78$ iterations, we obtain a graph topology
with a spectral pseudodistance very close to zero (in particular, $1.7e-3$)
and an eigenvalue spectrum remarkably similar to that of the small-world
network, as shown in Fig.~\ref{fig:CDF_small_world}.
\end{example}

\begin{example}[Power-Law]
\label{exmp:power_law}Another popular model in the `Network Science'
literature is the scale-free network. This model was proposed by Barab\'{a}%
si and Albert in \cite{BA99} to explain the presence of heavy-tailed degree
distributions in many real-world networks. In this example, we generate a
random power-law network with $n=50$ nodes and $m=4$, where $m$ is a
parameter that characterizes the average degree of the resulting network
(see \cite{BA99} for more details about this model). A random realization of
this network presents the following sequence of moments: $\left(
m_{k}\right) _{k=1}^{5}=(7.72,111,2.81e3,9.70e4,3.82e6)$. Then, after
running our algorithm for $98$ iterations, we obtain a graph topology with a
spectral pseudodistance very close to zero (in particular, $5.5e-2$). The
eigenvalue spectrum of the resulting topology is remarkably similar to that
of the small-world network, as shown in Fig.~\ref{fig:CDF_power_law}.
Furthermore, we can compare the degree sequences of the power-law network
and the topology generated by our algorithm. We compare these sequences,
sorted in descending order, in Fig.\ref{Fig:Degree_powerlaw}. We observe how
the degree sequence of the topology obtained in our algorithm is remarkably
close to that of the power-law network. This indicates that the spectral
properties of a network contains rich information about the network
structure, in particular, the first five spectral moments seems to highly
constrain many relevant structural properties of the graph, such as the
degree distribution.
\end{example}

\begin{figure}[tbp]
\centering\includegraphics[width=0.7\linewidth]{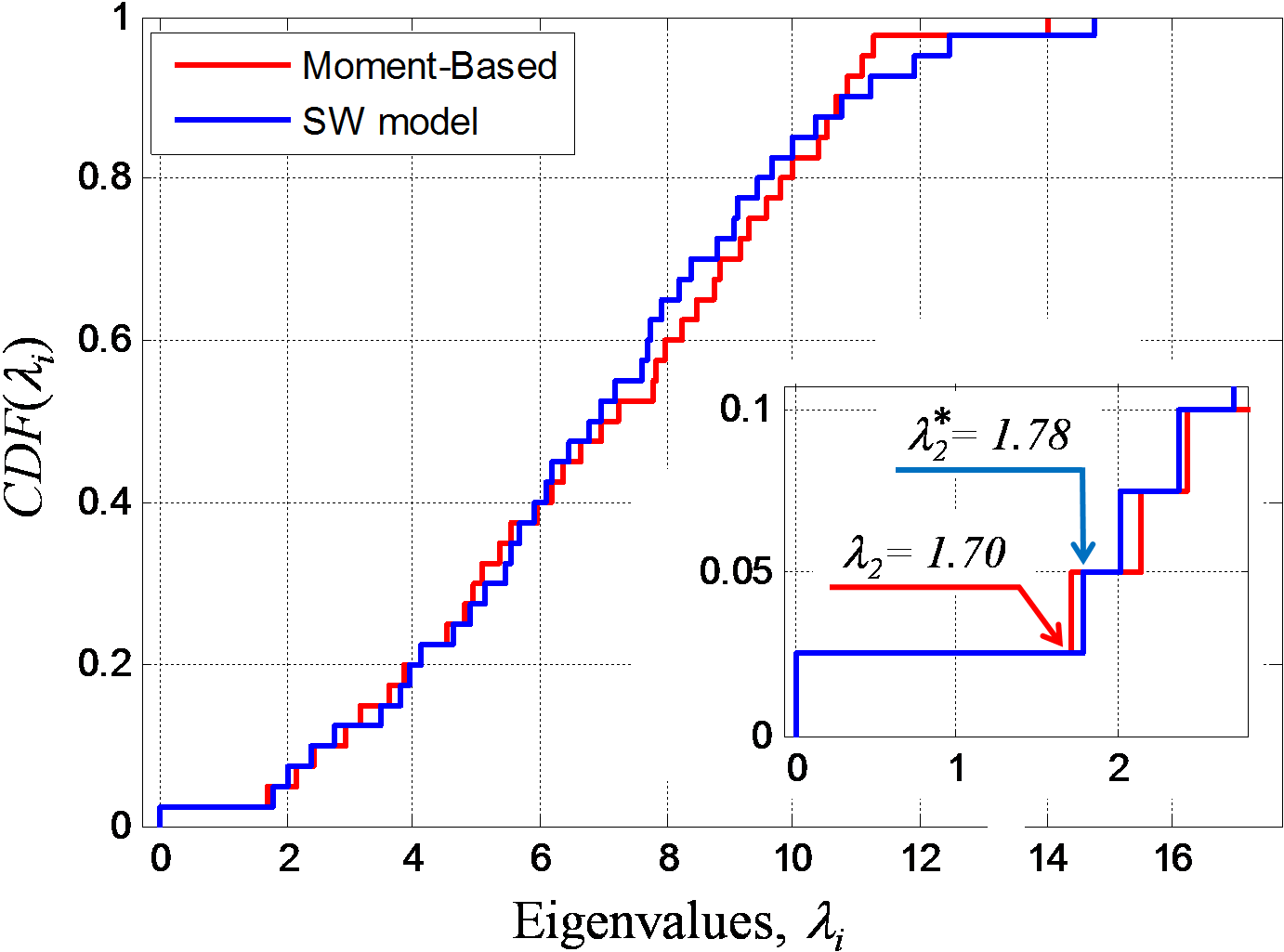}
\caption{Empirical cumulative distribution for the eigenvalue spectrum of
the small-world graph in Example \protect\ref{exmp:small_world} (blue) and
the topology resulting from our algorithm (red).}
\label{fig:CDF_small_world}
\end{figure}

\begin{figure}[tbp]
\centering\includegraphics[width=0.7\linewidth]{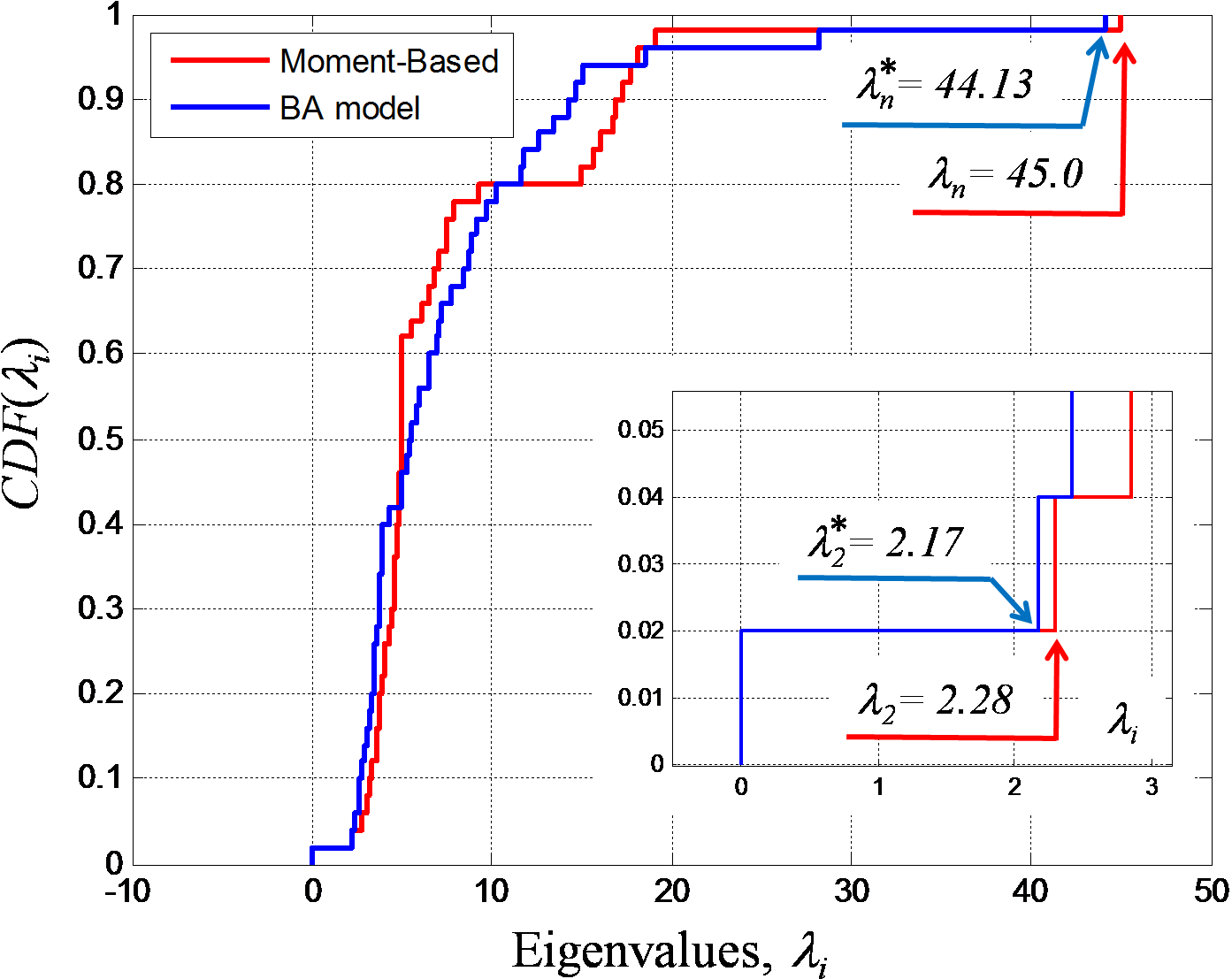}
\caption{Empirical cumulative distribution for the eigenvalue spectrum of
the power-law graph in Example \protect\ref{exmp:power_law} (blue) and the
topology resulting from our algorithm (red).}
\label{fig:CDF_power_law}
\end{figure}

\begin{figure}[tbp]
\centering\includegraphics[width=0.7\linewidth]{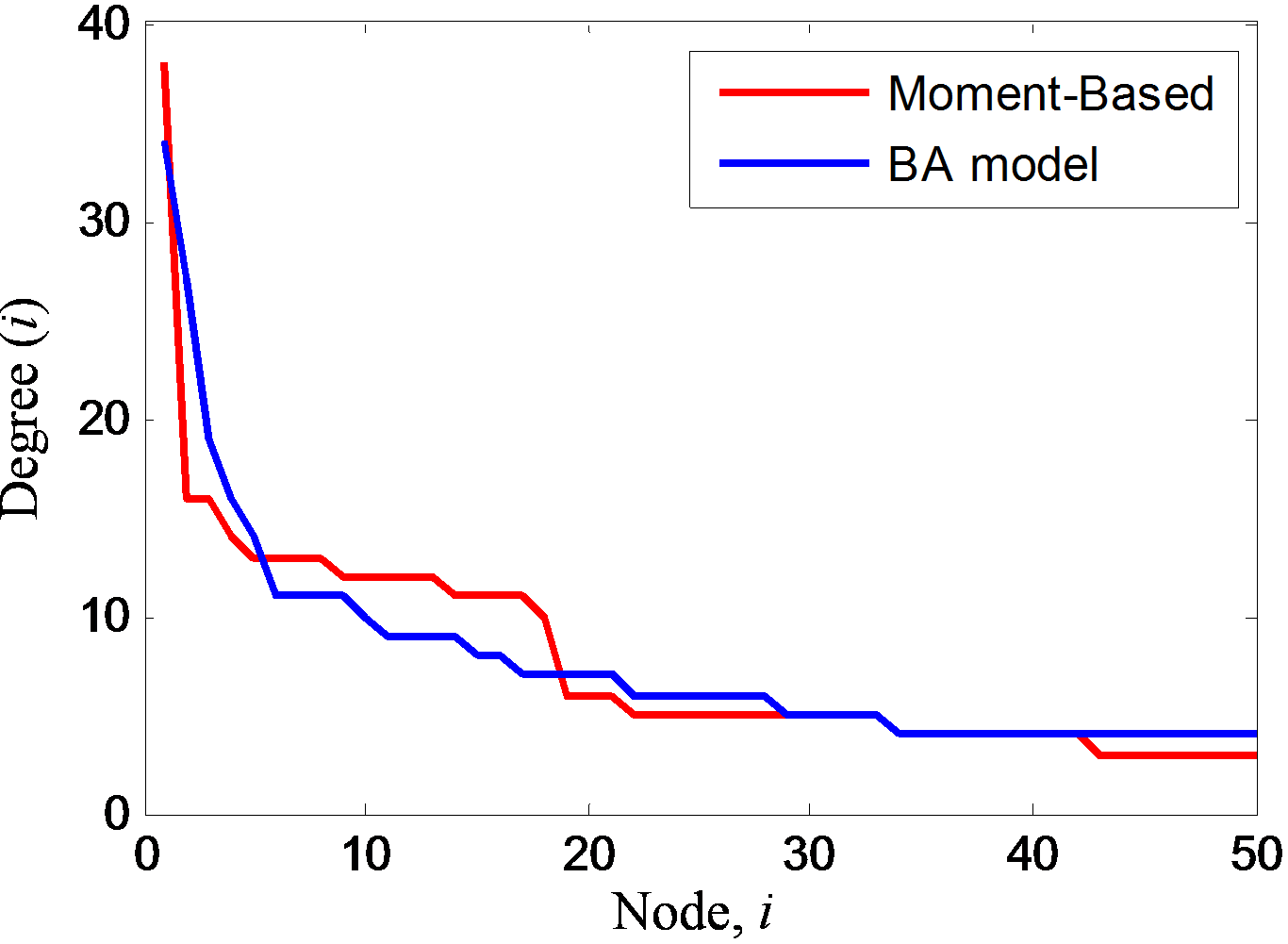}
\caption{Degree sequences (in descending order) of the power-law graph in
Example \protect\ref{exmp:power_law} (blue) and the topology resulting from
our algorithm (red).}
\label{Fig:Degree_powerlaw}
\end{figure}


\section{Conclusions and Future Research}

In this paper, we have described a fully decentralized algorithm that
iteratively modifies the structure of a network of agents with the objective
of controlling the spectral moments of the Laplacian matrix of the network.
Although we assume that each agent has access to local information regarding
the graph structure, we show that the group is able to collectively
aggregate their local information to take a global optimal decision. This
decision corresponds to the most beneficial link addition/deletion in order
to minimize a distance function that involves the Laplacian spectral moments
of the network. The aggregation of the local information is achieved via
gossip algorithms, which are also used to ensure network connectivity
throughout the evolution of the network.

Future work involves identifying sets of spectral moments that are reachable
by our control algorithm. (We say that a sequence of spectral moments is
reachable if there exists a graph whose moments match the sequence of
moments.) Furthermore, we observed that fitting a set of low-order moments
does not guarantee a good fit of the complete distribution of eigenvalues.
In fact, there are important spectral parameters, such as the algebraic
connectivity, that are not captured by a small set of spectral moments.
Nevertheless, we observed in numerical simulations that fitting the first
four moments of the eigenvalue spectrum often achieves a good reconstruction
of the complete spectrum. Hence, a natural question is to describe the set
of graphs most of whose spectral information is contained in a relatively
small set of low-order moments.

\bigskip

\appendix

\section{Proof of Theorem \protect\ref{Uniqueness from Moments}}

\begin{theorem}
\label{Uniqueness from Moments copy(1)}Consider two undirected (possibly
weighted) graphs $G_{1}$ and $G_{2}$ with (real) eigenvalue spectra $S\left(
G_{1}\right) =\{\lambda _{1}^{\left( 1\right) }\leq ...\leq \lambda
_{n}^{\left( 1\right) }\}$ and $S_{2}\left( G_{1}\right) =\{\lambda
_{1}^{\left( 2\right) }\leq ...\leq \lambda _{n}^{\left( 2\right) }\}$.
Then, $\lambda _{i}^{\left( 1\right) }=\lambda _{i}^{\left( 2\right) }$ for
all $1\leq i\leq n$ if and only if $m_{k}\left( G_{1}\right) =m_{k}\left(
G_{2}\right) $ for $0\leq k\leq n-1$.
\end{theorem}

\begin{proof}
The theorem states that the spectrum $S\left( A\right) =\left\{ \lambda
_{i}\right\} _{i=1}^{n}$ of any $n\times n$ symmetric matrix $A$ is uniquely
characterized by its first $n-1$ spectral moments. First, we use
Cayley-Hamilton theorem to prove that the first $n-1$ spectral moments of
the spectrum $S$ characterize the whole infinite sequence of moments $\left(
m_{k}\left( S\right) \right) _{k=0}^{\infty }$, as follows. Let $\phi \left(
\lambda \right) \triangleq \det \left( \lambda I_{n}-A\right) =\lambda
^{n}+\alpha _{n-1}\lambda ^{n-1}+...+\alpha _{0}$, be the characteristic
equation of $A$. Then, from Cayley-Hamilton, we have $\phi \left( A\right)
=0 $. Multiplying $\phi \left( A\right) $ by $\frac{1}{n}A^{t}$, and
applying the trace operator, we have that,%
\begin{eqnarray*}
\frac{1}{n}\text{Trace}\left[ A^{t}\phi \left( A\right) \right] &=&\frac{1}{n%
}\text{Trace}\left( A^{t+n}\right) +\alpha _{n-1}\frac{1}{n}\text{Trace}%
\left( A^{t+n-1}\right) +...+\alpha _{0}\frac{1}{n}\text{Trace}\left(
A^{t}\right) \\
&=&m_{t+n}\left( A\right) +\alpha _{n-1}m_{t+n-1}\left( A\right) +...+\alpha
_{0}m_{t}\left( A\right) =0,
\end{eqnarray*}%
for all $t\in \mathbb{N}$. Therefore, given the sequence of moments $\left(
m_{k}\left( A\right) \right) _{k=0}^{n-1}$, we can use the recursion%
\begin{equation*}
m_{t+n}\left( A\right) =-\alpha _{n-1}m_{t+n-1}\left( A\right) -...-\alpha
_{1}m_{t+1}\left( A\right) -\alpha _{0}m_{t}\left( A\right) ,
\end{equation*}%
to uniquely characterize the infinite sequence of moments $\left(
m_{k}\left( A\right) \right) _{k=0}^{\infty }$.

Second, we prove that the infinite sequence of moments $\left( m_{k}\left(
A\right) \right) _{k=0}^{\infty }$ uniquely characterizes the eigenvalue
spectrum. Let us define the \emph{spectral measure} of the matrix $A$ with
real eigenvalues $\lambda _{1}\leq \lambda _{2}\leq ...\leq \lambda _{n}$, as%
\begin{equation*}
\mu _{A}\left( x\right) =\sum_{i=1}^{n}\delta \left( x-\lambda _{i}\right) ,
\end{equation*}%
where $\delta \left( \bullet \right) $ is the Dirac delta function. In what
follows, we prove that the spectral measure of $A$ is uniquely characterized
by its infinite sequence of spectral moments using Carleman's condition \cite%
{Akh65}. Since there is a trivial bijection between the eigenvalue spectrum
of $A$ and its spectral measure, uniqueness of the spectral measure imply
uniqueness of the eigenvalue spectrum.

Carleman's condition states that a measure $\mu $ on $\mathbb{R}$ is
uniquely characterized by its infinite sequence of moments $\left(
M_{k}\left( \mu \right) \right) _{k=1}^{\infty }$ if (\emph{i}) $M_{k}\left(
\mu \right) <\infty $ for all $k\in \mathbb{N}$, and (\emph{ii}) 
\begin{equation*}
\sum_{s=1}^{\infty }\left( M_{2s}\left( \mu \right) \right) ^{-1/2s}=\infty .
\end{equation*}%
In our case, the moments of the spectral measure $\mu _{A}$ are%
\begin{eqnarray*}
M_{k}\left( \mu _{A}\right) &=&\int_{-\infty }^{+\infty }x^{k}d\mu
_{A}\left( x\right) \\
&=&\sum_{i=1}^{n}\lambda _{i}^{k}=n~m_{k}\left( A\right) .
\end{eqnarray*}%
These moments satisfy: (\emph{i}) $M_{k}\left( \mu _{A}\right) \leq n\lambda
_{n}^{k}<\infty $, for any finite matrix $A$, and (\emph{ii})%
\begin{eqnarray*}
\sum_{s=1}^{\infty }\left( M_{2s}\left( \mu _{A}\right) \right) ^{-1/2s}
&=&\sum_{s=1}^{\infty }\left( \sum_{i=1}^{n}\lambda _{i}^{2s}\right) ^{-1/2s}
\\
&\geq &\sum_{s=1}^{\infty }\left( \lambda _{n}^{2s}\right) ^{-1/2s} \\
&=&\sum_{s=1}^{\infty }\lambda _{n}^{-1}=\infty ,
\end{eqnarray*}%
for any $A\neq 0$. As a consequence, the spectral measure of any finite
matrix $A\neq 0$ with real eigenvalues is uniquely characterized by $\left(
M_{k}\left( \mu _{A}\right) \right) _{k=0}^{\infty }$. Since, $M_{k}\left(
\mu _{A}\right) =n~m_{k}\left( A\right) $, we have that the sequence of
moments $\left( m_{k}\left( A\right) \right) _{k=0}^{n-1}$ uniquely
characterizes $\left( M_{k}\left( \mu _{A}\right) \right) _{k=0}^{\infty }$.
Therefore, the sequence of moments $\left( m_{k}\left( A\right) \right)
_{k=0}^{n-1}$ uniquely characterize the spectral measure $\mu _{A}$ and the
real eigenvalue spectrum $S=\left\{ \lambda _{i}\right\} _{i=1}^{n}$.
\end{proof}


\end{document}